\documentclass[10pt]{amsart}
\usepackage{hyperref}
\usepackage{amsmath}
\usepackage{amssymb}
\usepackage{xypic}
\usepackage{yhmath}
\usepackage{setspace}
\def\beqnn{\begin{eqnarray*}}\def\eeqnn{\end{eqnarray*}}

\newtheorem{theorem}{Theorem}[section]
\newtheorem{theorem*}{Theorem}
\newtheorem{lemma}[theorem]{Lemma}
\newtheorem{sub-lemma}[theorem]{Sub-Lemma}
\newtheorem{proposition}[theorem]{Proposition}
\newtheorem{corollary}[theorem]{Corollary}
\newtheorem{claim}[theorem]{Claim}

\newtheorem{conjecture}[theorem]{Conjecture}

\theoremstyle{definition}    
\newtheorem{definition}[theorem]{Definition}

\newtheorem{remark}[theorem]{Remark}
\newtheorem{question}[theorem]{Question}

\numberwithin{equation}{section}

\begin{document}

\begin{center}
\title[The norms of multiplication operators]{On the norms of the multiplication operators between weighted Bergman spaces}
\end{center}

\author{Jianjun Jin}
\address{School of Mathematics Sciences, Hefei University of Technology, Xuancheng Campus, Xuancheng 242000, P.R.China}
\email{jin@hfut.edu.cn, jinjjhb@163.com}



\begin{abstract}
In this paper, we study the norms of multiplication operators acting between weighted Bergman spaces. First, we provide a proof for a norm estimate previously announced in our recent paper \cite{Jin-c}. Second, we establish a sharp norm estimate for certain special multiplication operators between weighted Bergman spaces, a result that is novel to the literature. Finally, we also discuss the connections between the Brennan conjecture and related multiplier norms induced by the Schwarzian derivative of univalent functions.
\end{abstract}

\maketitle

\section{{\bf {Introduction and main results}}}
Let $\Delta=\{z:|z|<1\}$ denote the unit disk in the
complex plane $\mathbb{C}$. For a domain $\Omega$ in $\mathbb{C}$, we use $\mathcal{A}(\Omega)$ to denote the class of all analytic functions in $\Omega$. For $-1<\alpha<\infty$, we define the weighted Bergman space ${\bf A}_{\alpha}^2={\bf A}_{\alpha}^2(\Delta)$ as
$${\bf A}_{\alpha}^2(\Delta)=\{\phi \in \mathcal{A}(\Delta) : \|\phi\|_{\alpha}^2:=(\alpha+1)\iint_{\Delta}|\phi (z)|^2(1-|z|^2)^{\alpha}\frac{dxdy}{\pi}<\infty\}.$$
Let $g \in \mathcal{A}(\Delta)$. The {\em multiplication operator} $M_g$ is defined by
$$M_g(\phi)(z):=g(z)\phi(z), \,\,\phi\in \mathcal{A}(\Delta).$$
For $\gamma\in \mathbb{R}$, we define the growth space $\mathcal{B}_{\gamma}=\mathcal{B}_{\gamma}(\Delta)$ as
$$
\mathcal{B}_{\gamma}(\Delta)=\left\{ \phi \in \mathcal{A}(\Delta) : 
\|\phi\|_{\mathcal{B}_{\gamma}}:=\sup_{z\in \Delta}|\phi(z)|(1-|z|^2)^{\gamma}<\infty \right\},
$$
which is a Banach space with the norm $\|\cdot\|_{\mathcal{B}_{\gamma}}$ defined above.
It is well known that 
\begin{theorem}\label{th-1}Let $\beta>\alpha>-1$. 
The multiplication operator $M_g$ is bounded from ${\bf A}_{\alpha}^2$ to ${\bf A}_{\beta}^2$ if and only if $g$ belongs to $\mathcal{B}_{\gamma}$ with $\gamma=(\beta-\alpha)/2.$
\end{theorem}
\begin{remark}
A proof of this result can be found in \cite{zhao}. 
\end{remark}
For $\beta>\alpha>-1$, when $M_g$ is bounded from ${\bf A}_{\alpha}^2$ to ${\bf A}_{\beta}^2$, we say $g$ is a {\em multiplier} from ${\bf A}_{\alpha}^2$ to ${\bf A}_{\beta}^2$. The class of all multipliers, denoted by $\mathbf{M}_{\alpha, \beta}$, is a Banach space with the following multiplier norm
$$
\|g\|_{{\bf {M}}_{\alpha, \beta}}:=\sup_{\phi \in {\bf A}^{2}_{\alpha}, \phi \neq 0}\frac{\|g \phi\|_{\beta}}{\|\phi\|_{\alpha}}.
$$ 

In the paper \cite{SS}, Shimorin studied a special multiplication operator from ${\bf A}_{\alpha}^2$ to ${\bf A}_{\alpha+4}^2$, which is induced by the Schwarzian derivative of a univalent function.  To recall the results of \cite{SS}, we first introduce some notations and definitions. We let $\mathcal{S}$ be the class of all univalent functions $f$ in $\Delta$ with $f(0)=f'(0)-1=0$.  

For a locally univalent function $f$ in an open domain $\Omega$ of $\mathbb{C}$, the {\em Schwarzian derivative} $S(f)$ of $f$ is defined by 
$$S(f)(z)=\bigg[\frac{f''(z)}{f'(z)}\bigg]'-\frac{1}{2}\left[\frac{f''(z)}{f'(z)}\right]^2, \,z\in \Omega.$$
Here, $N(f)(z):= f''(z)/f'(z)$ is called the Pre-Schwarzian derivative of $f$. The following  multiplication operator has been considered by Shimorin in \cite{SS}. 
$$M_{S(f)}(\phi)(z):=S(f)(z)\phi(z),\, \phi\in \mathcal{A}(\Delta).$$
It is well known that
 $$|S(f)|(1-|z|^2)^2\leq 6,$$
 for all $f\in \mathcal{S}$, and hence $S(f)$ belongs to $\mathcal{B}_2$, see \cite{Du}. Then $S(f)$ is a multiplier from ${\bf A}_{\alpha}^2$ to ${\bf A}_{\alpha+4}^2$ for any $f\in \mathcal{S}$. 
Let  $t\in \mathbb{R}$. For $f\in \mathcal{S}$, the {\em integral means spectrum} $\beta_f(t)$ is defined by
\begin{equation}\label{defi}
\beta_f(t):=\limsup_{r \rightarrow 1^{-}}\frac{\log \int_{-\pi}^{\pi} |f'(re^{i\theta})|^{t}d\theta}{|\log(1-r)|}. 
\end{equation}
The {\em universal integral means spectrum} $B(t)$ is defined by
\begin{equation}
B(t)=\sup\limits_{f\in \mathcal{S}} \beta_f(t).\nonumber
\end{equation}
It was proved in \cite{SS} that
\begin{theorem}
Let $\alpha>-1$. Then
$$\|S(f)\|_{{\bf {M}}_{\alpha, \alpha+4}}\leq 36\frac{\alpha+3}{\alpha+1},$$
holds for all $f\in \mathcal{S}$. 
\end{theorem}
By using this result, Shimorin obtained better bounds for the universal integral means spectrum $B(t)$ for $t=-1$ and $t=-2$. $B(-2)=1$ is known as the famous Brennan conjecture, see \cite{GM}. Later, following this line of thought, Hedenmalm and Shimorin established in \cite{HS-1} the current best upper bound estimate about $B(-2)$ and other universal integral means spectra. For more results about the Brennan conjecture and universal integral means spectrum, see \cite{Be}, \cite{HSo}.

For simplicity, in the sequel, for a univalent function $f\in \mathcal{S}$, we will write 
$$\mathbb{M}_f(\alpha):=\|S(f)\|_{{\bf M}_{\alpha, \alpha+4}}.$$
For a multiplier $g$ from ${\bf A}_{\alpha}^2$ to ${\bf A}_{\beta}^2$ with $\beta>\alpha>-1$, we will write
$$\mathbb{M}_g(\alpha, \beta)=\mathbb{M}_{g(z)}(\alpha, \beta):=\|g\|_{{\bf M}_{\alpha, \beta}}.$$
There is a close connection between the norm of the multiplication operator ${M}_{S(f)}$ and the Brennan conjecture. 
It has been shown in \cite{SS} that 
 \begin{proposition}\label{last-p}
${\bf (1)}$ Let $\alpha>-1$. We have
\begin{equation} \mathbb{M}_{\kappa}^2(\alpha)=\frac{36(\alpha+3)(\alpha+5)}{(\alpha+2)(\alpha+4)},\nonumber \end{equation}
Here, $\kappa$ is the well-known Koebe function, which is defined by
 $$\kappa(z):=\frac{z}{(1-z)^2}, z\in \Delta.$$

${\bf (2)}$ If some univalent function $f\in \mathcal{S}$ satisfies that
\begin{equation} \mathbb{M}_f^2(\alpha) \leq\mathbb{M}_{\kappa}^2(\alpha)=\frac{36(\alpha+3)(\alpha+5)}{(\alpha+2)(\alpha+4)},\nonumber \end{equation}
 for any $\alpha>0$. Then $\beta_f(-2)\leq 1$, that is to say, the Brennan conjecture is true for the function $f$. 
 \end{proposition}
 
Let $\mathcal{U}_R$ be the class of all univalent rational functions $f$ with $f\in\mathcal{S}$. In the recent paper \cite{Jin-c}, we have proved that
if $R\in \mathcal{U}_R$ has at least one critical point on the unit circle $\mathbb{T}$, i.e., there is at least one point $z_0$ on $\mathbb{T}$ such that $R'(z_0)=0$, then $\beta_R(-2)=1$. 
On the other hand, by checking carefully the arguments in the proof of Proposition $8$ in \cite[Page 1632]{SS}, we see that, if there is a constant $\varepsilon>0$ such that $$\mathbb{M}_f(\alpha)+\mathbb{M}_f(\alpha+2)\leq \mathbb{M}_{\kappa}(\alpha)+\mathbb{M}_{\kappa}(\alpha+2),$$
for $\alpha\in (0,\varepsilon)$, then $\beta_f({-2})\leq 1$. Moreover, for $R\in \mathcal{U}_R$ that have at least one critical point on $\mathbb{T}$, if there are two small constants $\epsilon>0, \eta>0$ such that $\mathbb{M}_R(\alpha)<\mathbb{M}_{\kappa}(\alpha)-\eta$
for $\alpha \in (-\epsilon, \epsilon)\cup (2-\epsilon,2+\epsilon)$, then we will obtain that $\beta_{R}(-2)<1$, which is a contradiction. This leads us to conjecture the following result holds. 
The first aim of this paper is to give a proof for the following theorem. 

\begin{theorem}\label{m-1}
For each $\alpha>-1$, we have $\mathbb{M}_R(\alpha)\geq \mathbb{M}_{\kappa}(\alpha)$ for all $R\in \mathcal{U}_R$ that have at least one critical point on $\mathbb{T}.$ 
\end{theorem}

In view of Theorem \ref{th-1}, two natural questions are the following
\begin{question} For a multiplier $g$ with $\|g\|_{{\mathcal B}_{(\beta-\alpha)/2}}=1$, what is its exact multiplier norm? \end{question}
\begin{question} What is the exact value of 
$$\mathbb{M}(\alpha, \beta):=\sup_{\|g\|_{{\mathcal B}_{(\beta-\alpha)/2}}=1}\|g\|_{{\bf M}_{\alpha, \beta}}?$$\end{question}

We will partially answer these two questions. We shall prove that
 \begin{theorem}\label{m-2}
Let $\beta>\alpha>-1$. Define $g_0(z):=(1-z^2)^{-\frac{\beta-\alpha}{2}}, z\in \Delta$. Then we have 
$$\mathbb{M}_{g_0}^2(\alpha, \beta)=\frac{1}{2^{\beta-\alpha}}\frac{\Gamma(\beta+2)}{\Gamma(\alpha+2)}\Big[\frac{\Gamma(1+\alpha/2)}{\Gamma(1+\beta/2)}\Big]^2.$$
Here $\Gamma$ is the usual Gamma function, see \cite{AAR}. 
\end{theorem}
\begin{remark}To the best of our knowledge, Theorem \ref{m-2} is novel in the literature. Also, note that $$S(\kappa)(z)=-\frac{6}{(1-z^2)^2}. $$Then the first part of Proposition \ref{last-p} follows by taking $\beta=\alpha+4$ in Theorem \ref{m-2}. 
\end{remark}
From Theorem \ref{m-2}, we obtain that  
\begin{corollary}Let $\beta>\alpha>-1$. Then
$$\mathbb{M}^2(\alpha, \beta)\geq \frac{1}{2^{\beta-\alpha}}\frac{\Gamma(\beta+1)}{\Gamma(\alpha+1)}\Big[\frac{\Gamma(1+\alpha/2)}{\Gamma(1+\beta/2)}\Big]^2.$$
\end{corollary}

The paper is organized as follows. In the next section, we will establish some auxiliary results which will be used in our proof of main results in the paper. We will prove Theorem \ref{m-1} and \ref{m-2} in Section 3 and 4, respectively. In Section 5, we present some remarks.  

\section{{\bf Some Auxiliary Results}}
In this section, we present some auxiliary results which are needed in our proof of main theorems of this paper. First, we need the following integral representation of weighted Bergman spaces. Let $\mathbb{C}_{+} = \{w=u+iv: v>0\}$ be the upper half-plane and $\mathbb{R}_{+}=(0, \infty)$ be the set of positive real numbers. For $\alpha>-1$, we define the weighted Bergman space on the upper half-plane $\bf {A}_{\alpha}^2(\mathbb{C}_{+})$ as 
$$
{\bf {A}}_{\alpha}^2(\mathbb{C}_{+}) = \left\{ h\in \mathcal{A}(\mathbb{C}_{+}):  \|h\|_{\alpha}^{2}=\iint_{\mathbb{C}_{+}} |h(u + iv)|^{2} v^{\alpha} \frac{dudv}{\pi} < \infty \right\}.
$$
The following lemma will be useful. 
\begin{lemma}\label{l-1}Let $\alpha>-1$. Then $h$ belongs to ${\bf {A}}_{\alpha}^2(\mathbb{C}_{+})$ if and only if there is a function $\tilde{h} \in L_{\alpha}^{2}(\mathbb{R}_{+})$ such that
\begin{equation}\label{trans}
h(w) = \int_{0}^{\infty} \tilde{h}(t) e^{i w t} \, dt, \,\,w \in \mathbb{C}_{+},
\end{equation}
and we have the following norm identity 
$$
\|h\|_{\alpha}^{2} = \frac{\Gamma(\alpha+1)}{2^{\alpha}} \int_{0}^{\infty} |\tilde{h}(t)|^{2} \frac{dt}{t^{\alpha+1}}.
$$
Here $L_{\alpha}^{2}(\mathbb{R}_{+})$ denotes the weighted Lebesgue space consisting of all measurable functions $\tilde{h}$ on $\mathbb{R}_{+}$ with the norm
$$
\|\tilde{h}\|_{L_{\alpha}^{2}}^2=\int_{0}^{\infty} |\tilde{h}(t)|^{2} \frac{dt}{t^{\alpha+1}}.
$$
\end{lemma}
\begin{remark}
${\bf {A}}_{\alpha}^2(\mathbb{C}_{+})$ is isometrically isomorphic to the space $L_{\alpha}^{2}(\mathbb{R}_{+})$ under the Fourier-Laplace type transform (\ref{trans}). See \cite{DGM} for a proof of this lemma.\end{remark}
We will need the following estimates.
\begin{lemma}\label{l-2}
Let $\alpha>-1, r\in (0,1), \lambda>0$. If $2\lambda>\alpha+2$, then  
\begin{equation}\label{a-e-1}
 \|{(1-rz^2)^{-\lambda}}\|_{\alpha}^2=\frac{\Gamma(\alpha+2)\Gamma(2\lambda-\alpha-2)}{2^{\alpha+1}[\Gamma(\lambda)]^2}\cdot\frac{1+{o}(1)}{(1-r^2)^{2\lambda-\alpha-2}},\, {\textup{as}}\,\, r\rightarrow 1^{-}. 
 \end{equation}
 and \begin{equation}\label{a-e-2}
   \|{(1-rz)^{-\lambda}}\|_{\alpha}^2=\frac{\Gamma(\alpha+2)\Gamma(2\lambda-\alpha-2)}{[\Gamma(\lambda)]^2}\cdot\frac{1+{o}(1)}{(1-r^2)^{2\lambda-\alpha-2}},\, {\textup{as}}\,\, r\rightarrow 1^{-}. 
 \end{equation} 
 If $0<2\lambda<\alpha+2$, then there is a constant $M>0$ such that for all $r\in (0,1)$,
 \begin{equation}\label{a-e-3}
   \|{(1-rz)^{-\lambda}}\|_{\alpha}^2\leq M.
 \end{equation}\end{lemma} 
\begin{remark}Here and in the sequel, ${o}(1)$ denotes a quantity depending on $r$ such that ${o}(1)\rightarrow 0$ as $r\rightarrow 1^{-}$, and which may vary from line to line.\end{remark}

\begin{proof}[Proof of Lemma \ref{l-2}]
First, for $\gamma>0$, we have 
\begin{equation}\label{l-eq-1}
  \frac{1}{(1-z)^{\gamma}}=\sum_{n=0}^{\infty}\frac{\Gamma(n+\gamma)}{n!\Gamma(\gamma)}z^n, z\in \Delta.\end{equation}
Also, for any $\phi=\sum_{n=0}^{\infty}a_nz^n\in \mathbf{A}_{\alpha}^2(\Delta)$, we have
\begin{equation}\label{l-eq-2}\|\phi\|_{\alpha}^2=\sum_{n=0}^{\infty}\frac{n!\Gamma(\alpha+2)}{\Gamma(n+\alpha+2)}|a_n|^{2}.
\end{equation}
For $r\in (0,1)$ and $\lambda>0$, from (\ref{l-eq-1}) and (\ref{l-eq-2}), we obtain that
  \begin{equation}\label{l-eq-3}
 \|{(1-rz^2)^{-\lambda}}\|_{\alpha}^2=\sum_{n=0}^{\infty}\frac{(2n)!\Gamma(\alpha+2)}{\Gamma(2n+\alpha+2)}\Big|\frac{\Gamma(n+\lambda)}{n!\Gamma(\lambda)}\Big|^2 r^{2n}, 
 \end{equation}
and 
 \begin{equation}\label{l-eq-4}
 \|{(1-rz)^{-\lambda}}\|_{\alpha}^2=\sum_{n=0}^{\infty}\frac{n!\Gamma(\alpha+2)}{\Gamma(n+\alpha+2)}\Big|\frac{\Gamma(n+\lambda)}{n!\Gamma(\lambda)}\Big|^2 r^{2n}.
 \end{equation}
By Stirling's formula, we have
 \begin{equation}
 \frac{\Gamma(n+\lambda)}{n!}=(n+1)^{\lambda-1}[1+\widetilde{o}(1)],\,\, {\textup{as}} \,\, n\rightarrow \infty.\nonumber 
 \end{equation}
 Here and later, we use $\widetilde{o}(1)$ to denote a sequence that tends to $0$, i.e., $o(1)\rightarrow 0, \, {\textup{as}}\,\, n \rightarrow \infty$, which may be different in different places. Then it follows from (\ref{l-eq-3}) that 
 \begin{eqnarray}\label{l-eq-5}
 \|{(1-rz^2)^{-\lambda}}\|_{\alpha}^2&=&\sum_{n=0}^{\infty}\frac{\Gamma(\alpha+2)}{[2(n+1)]^{\alpha+1}}[1+\widetilde{o}(1)]\cdot\frac{(n+1)^{2\lambda-2}}{[\Gamma(\lambda)]^2}[1+\widetilde{o}(1)]r^{2n}\nonumber \\
 &=&\frac{1}{2^{\alpha+1}}\frac{\Gamma(\alpha+2)}{[\Gamma(\lambda)]^2}\sum_{n=0}^{\infty}{(n+1)^{2\lambda-\alpha-3}}[1+\widetilde{o}(1)]r^{2n}.
 \end{eqnarray}
Similarly, from (\ref{l-eq-4}), we have 
 \begin{eqnarray}\label{l-eq-6}
 \|{(1-rz)^{-\lambda}}\|_{\alpha}^2=\frac{\Gamma(\alpha+2)}{[\Gamma(\lambda)]^2}\sum_{n=0}^{\infty}{(n+1)^{2\lambda-\alpha-3}}[1+\widetilde{o}(1)]r^{2n}.
 \end{eqnarray}
Note that there is a constant $C>0$ such that for $a<-1$,
$$\sum_{n=1}^{\infty}n^a r^{2n}\leq C,\,\,{\text{for all}}\,\, r\in (0,1).$$
Consequently, when $0<2\lambda<\alpha+2$ so that $2\lambda-\alpha-3<-1$, we see that (\ref{a-e-3}) holds. 

On the other hand, when $2\lambda>\alpha+2$, we obtain from Stirling's formula and \ref{l-eq-1} that 
  \begin{eqnarray}\label{l-eq-7}
 \lefteqn{\sum_{n=0}^{\infty}{(n+1)^{2\lambda-\alpha-3}}[1+\widetilde{o}(1)]r^{2n}=\sum_{n=0}^{\infty}\frac{\Gamma(n+2\lambda-\alpha-2)}{n!\Gamma(2\lambda-\alpha-2)}[1+\widetilde{o}(1)]r^{2n}}\nonumber \\
 &&=\frac{1}{{(1-r^2)^{2\lambda-\alpha-2}}}+\sum_{n=0}^{\infty}\frac{\Gamma(n+2\lambda-\alpha-2)}{n!\Gamma(2\lambda-\alpha-2)}\cdot [\widetilde{o}(1)]\cdot r^{2n}. 
 \end{eqnarray}
Let $\widetilde{o}(1):=b_n$. Note that for any $\varepsilon\in (0,1/2)$ there exist two constants $C_0>0, N_0\in \mathbb{N}$ such that $|b_n|\leq C_0$ for all $n\in \mathbb{N}_0$ and $|b_n|\leq \varepsilon/2$ for all $n\geq N_0$ in the last term of (\ref{l-eq-7}). 
It follows from again  (\ref{l-eq-1}) that
\begin{eqnarray}\label{l-eq-8}
\lefteqn{\sum_{n=0}^{\infty}\frac{\Gamma(n+2\lambda-\alpha-2)}{n!\Gamma(2\lambda-\alpha-2)}\cdot [\widetilde{o}(1)]\cdot r^{2n}} \nonumber \\
  && \leq C_0\sum_{n=0}^{N_0} \frac{\Gamma(n+2\lambda-\alpha-2)}{n!\Gamma(2\lambda-\alpha-2)}r^{2n}+\varepsilon/2\sum_{n={{n_0}+1}}^{\infty}\frac{\Gamma(n+2\lambda-\alpha-2)}{n!\Gamma(2\lambda-\alpha-2)}r^{2n} \nonumber \\
  && \leq C_0\sum_{n=0}^{N_0} \frac{\Gamma(n+2\lambda-\alpha-2)}{n!\Gamma(2\lambda-\alpha-2)}+\frac{\varepsilon/2}{{(1-r^2)^{2\lambda-\alpha-2}}}\nonumber \\
  &&:=C_1+\frac{\varepsilon/2}{{(1-r^2)^{2\lambda-\alpha-2}}}.
  \end{eqnarray}
Meanwhile, \begin{eqnarray}\label{l-eq-9}
\lefteqn{\sum_{n=0}^{\infty}\frac{\Gamma(n+2\lambda-\alpha-2)}{n!\Gamma(2\lambda-\alpha-2)}\cdot [\widetilde{o}(1)]\cdot r^{2n}} \nonumber \\
  && \geq \sum_{n=0}^{N_0} \frac{\Gamma(n+2\lambda-\alpha-2)}{n!\Gamma(2\lambda-\alpha-2)}b_n r^{2n}-\varepsilon/2\sum_{n={{n_0}+1}}^{\infty}\frac{\Gamma(n+2\lambda-\alpha-2)}{n!\Gamma(2\lambda-\alpha-2)}r^{2n} \nonumber \\
  && \geq -\sum_{n=0}^{N_0} \frac{\Gamma(n+2\lambda-\alpha-2)}{n!\Gamma(2\lambda-\alpha-2)}|b_n|-\frac{\varepsilon/2}{{(1-r^2)^{2\lambda-\alpha-2}}}\nonumber \\
  &&:=-C_2-\frac{\varepsilon/2}{{(1-r^2)^{2\lambda-\alpha-2}}}.
  \end{eqnarray}
Combining (\ref{l-eq-7}), (\ref{l-eq-8}) and (\ref{l-eq-9}), we obtain that
    \begin{eqnarray}
[{(1-r^2)^{2\lambda-\alpha-2}}]\sum_{n=0}^{\infty}{(n+1)^{2\lambda-\alpha-3}}[1+\widetilde{o}(1)]r^{2n}\leq 1+\varepsilon/2+C_1[{(1-r^2)^{2\lambda-\alpha-2}}], \nonumber
 \end{eqnarray}
 and 
     \begin{eqnarray}
[{(1-r^2)^{2\lambda-\alpha-2}}]\sum_{n=0}^{\infty}{(n+1)^{2\lambda-\alpha-3}}[1+\widetilde{o}(1)]r^{2n}\geq 1-\varepsilon/2-C_2[{(1-r^2)^{2\lambda-\alpha-2}}]. \nonumber
 \end{eqnarray}
Since $C_1$ and $C_2$ are independent of $r$, then we can find a constant $\delta\in (0,1)$ such that
$$C_1{(1-r^2)^{2\lambda-\alpha-2}}<\varepsilon/2,\,\,\text{and}\,\, C_2{(1-r^2)^{2\lambda-\alpha-2}}<\varepsilon/2,$$
for all $r\in (1-\delta, 1)$. Consequently, for any $\varepsilon\in (0,1/2)$, we see that 
    \begin{eqnarray}
1-\varepsilon\leq [{(1-r^2)^{2\lambda-\alpha-2}}]\sum_{n=0}^{\infty}{(n+1)^{2\lambda-\alpha-3}}[1+\widetilde{o}(1)]r^{2n}\leq 1+\varepsilon, \nonumber
 \end{eqnarray}
for all $r\in (1-\delta, 1)$. This is equivalent to  
 \begin{equation}\label{l-eq-10}
\sum_{n=0}^{\infty}{(n+1)^{2\lambda-\alpha-3}}[1+\widetilde{o}(1)]r^{2n}=\frac{1+{o}(1)}{(1-r^2)^{2\lambda-\alpha-2}},\, {\textup{as}}\,\, r\rightarrow 1^{-}. 
 \end{equation}
 Then, in view of (\ref{l-eq-10}), (\ref{a-e-1}) and (\ref{a-e-2}) respectively follow from (\ref{l-eq-5}) and (\ref{l-eq-6}).  The lemma is proved.
 \end{proof} 
The following key lemma will also be needed.
\begin{lemma}\label{lemma-key}Let $\beta>\alpha>-1$. If $g$ is a multiplier from ${\bf A}_{\alpha}^2$ to ${\bf A}_{\beta}^2$, then for any $a\in \Delta$, 
$$\mathbb{M}_{g(az)}(\alpha, \beta)\leq \mathbb{M}_{g}(\alpha, \beta).$$
\end{lemma}
\begin{remark}
This result has been used in the paper \cite{SS} for the case $\beta=\alpha+4$. For the convenience of the reader, we will provide a detailed proof of Lemma \ref{lemma-key}, which also clarifies some imprecise statements in \cite[page 2310, lines 17-20]{Jin-1}.
\end{remark}
\begin{proof}[Proof of Lemma \ref{lemma-key}]
Let $\gamma=(\beta-\alpha)/2>0$. We first show that if $g$ is a multiplier from ${\bf A}_{\alpha}^2$ to ${\bf A}_{\beta}^2$, then the map $\dag: a \mapsto g(az)$ is analytic from the unit disk into $\mathcal{B}_{\gamma}$.
Since $g \in \mathcal{B}_{\gamma}$, then for a fixed $a \in \Delta$ and any $z \in \Delta$, we have $1-|z|^2\leq 1-|az|^2$ so that 
$$
|g(az)|(1-|z|^2)^{\gamma} =|g(az)|(1-|az|^2)^{\gamma}\Big[\frac{1-|z|^2}{1-|az|^2}\Big]^{\gamma} \leq \|g\|_{\mathcal{B}_{\gamma}}.
$$
This yields that $g(az) \in \mathcal{B}_\gamma$ and the mapping $\dag$ is well-defined.

We will show that $\dag$ is analytic at any point $a_0 \in \Delta$. We define $\widetilde{g}(z) = z g'(a_0z)$ and let
\begin{equation}\label{lemma-k}
\mathbf{Q}_{a_0, \tau}(z):= \frac{g((a_0+\tau)z) - g(a_0z)}{\tau}.
\end{equation}
In the following, we assume that $\tau$ satisfies $|\tau|<1-|a_0|$ so that $(a_0+\tau)z|<1$ and $g((a_0+\tau)z$ converges in $\Delta$.
We will prove that \begin{equation}\label{equation}
\lim_{\tau \to 0} \|\mathbf{Q}_{a_0,\tau}-\widetilde{g}\|_{\mathcal{B}_\gamma} = 0.\end{equation}
Let $$g(z)=\sum_{n=1}^{\infty}b_nz^n.$$ Then we have $g'(z)=\sum_{n=1}^{\infty}nb_nz^{n-1}$. We first consider the case $a_0=0$. We have
\begin{eqnarray} 
\mathbf{Q}_{0,\tau}(z)-\widetilde{g}(z)&=&\frac{g(\tau z)-g(0)}{\tau}-zg'(0)\nonumber \\
&=&\frac{1}{\tau}[\sum_{n=0}^{\infty}b_n(\tau z)^n-b_0-\tau zb_1]\nonumber \\
&=&z^2\tau\sum_{n=2}^{\infty}b_n(\tau z)^{n-2}.\nonumber
\end{eqnarray}
It is easy to see that $\sum_{n=2}^{\infty}b_n z^{n-2}$
converges in $\Delta$. Then there are two constants $C_1>0, \delta\in (0,1)$ such that $|\sum_{n=2}^{\infty}b_n(\tau z)^{n-2}|<C_1
$ for all $z\in \Delta$ when $|\tau|<\delta$.  It follows that 
$z^2\tau\sum_{n=2}^{\infty}b_n(\tau z)^{n-2}$
converges uniformly to $0\,\, {\text as}\,\, \tau\to 0.$ 
This implies that (\ref{equation}) holds for $a_0=0$. 

When $a_0\neq 0$, we have
\begin{eqnarray}\label{remo-1} 
 \lefteqn{\mathbf{Q}_{a_0,\tau}(z)-\widetilde{g}(z)}
 \nonumber \\
 &&\quad=\frac{1}{\tau}\Big[\sum_{n=0}^{\infty}(1+\frac{\tau}{a_0})^nb_n(a_0z)^n-\sum_{n=0}^{\infty}b_n(a_0z)^n-\sum_{n=1}^{\infty}n\tau b_nz(a_0z)^{n-1}\Big]\nonumber \\
&&\quad=\frac{1}{\tau}\Big[\sum_{n=2}^{\infty}\sum_{k=2}^n \binom{n}{k}(\frac{\tau}{a_0})^{k}b_n(a_0 z)^{n}\Big]=\frac{\tau}{{a_0}^{2}}\Big[\sum_{n=2}^{\infty}\sum_{k=2}^n \binom{n}{k}(\frac{\tau}{a_0})^{k-2}b_n(a_0 z)^{n}\Big].
\end{eqnarray}
Since $|a_0|+|\tau|<1$, then we have $\lambda:=1-|a_0+\tau|>0$. Meanwhile, we know that 
$\sum_{n=0}^{\infty}(1+\frac{\tau}{a_0})^nb_n z^n$
converges in the disk centered at the origin with radius $R_1=1/|1+\frac{\tau}{a_0}|=\frac{|a_0|}{|a_0+\tau|}$. It follows from (\ref{remo-1}) that 
$$\sum_{n=2}^{\infty}\sum_{k=2}^n \binom{n}{k}(\frac{\tau}{a_0})^{k-2}b_n z^{n}$$ 
converges in the disk $\Delta_{R}$ centered at the origin with radius $R=\min\{1, R_1\}$. Note that $R_1=\frac{|a_0|}{1-\lambda}>|a_0|$ so that $a_0z$ is an interior point of the disk $\Delta_{R}$. 
Consequently, there is a constant $C_2>0$ such that $$|\sum_{n=2}^{\infty}\sum_{k=2}^n \binom{n}{k}(\frac{\tau}{a_0})^{k-2}b_n(a_0 z)^{n}|<C_2
$$
for all $z\in \Delta$.  From again (\ref{remo-1}), we see that 
$\mathbf{Q}_{a_0,\tau}(z)-\widetilde{g}(z)$
converges uniformly to $0\,\, {\text as}\,\, \tau\to 0.$ This proves that (\ref{equation}) holds for each $a_0\neq0$. 
This means that the mapping $\dag: a \mapsto g(az)$ is analytic from the unit disk to $\mathcal{B}_{\gamma}$.

Now, for a fixed $\phi \in \mathbf{A}_{\alpha}^2$, by using the above notations and arguments, we obtain that $\widetilde{g}\phi$ belongs to $\mathbf{A}_{\beta}^2$ and for each $a_0\in \Delta$,
$$
\|\mathbf{Q}_{a_0, \tau}\phi-\widetilde{g}\phi\|_{\beta}\leq \sqrt{(\beta+1)/{(\alpha+1)}} \|\mathbf{Q}_{a_0, \tau}-\widetilde{g}\|_{\mathcal{B}_{(\beta-\alpha)/2}}\|\phi\|_{\alpha},\nonumber
$$
when $\tau$ is close enough to $0$. It follows that 
$$
\lim_{\tau \to 0} \|\mathbf{Q}_{a_0, \tau}\phi-\widetilde{g}\phi\|_{\beta}=0,\,\,{\text{for each}\,\,a_0\in \Delta}.
$$
This proves that 
\begin{sub-lemma}\label{sub-1}
For a fixed $\phi \in \mathbf{A}_{\alpha}^2$, the mapping $\ddag: a\mapsto g(az)\phi(z)$ is analytic from the unit disk to $\mathbf{A}_{\beta}^2$. 
\end{sub-lemma} 

Next we shall prove that
\begin{sub-lemma}\label{sub-2}
For a fixed $\phi \in \mathbf{A}_{\alpha}^2$,  $$
\lim_{\Delta \ni a \to e^{i\theta}} \| g(a z)\phi(z) - g(e^{i\theta}z)\phi(z) \|_{\beta} = 0,$$
for any $\theta \in (-\pi,\pi]$.
\end{sub-lemma}  
We write $g_a(z)=g(az)$. As noted above, we obtain that $\|g_a\|_{\mathcal{B}_{\beta-\alpha}}\leq \|g\|_{\mathcal{B}_{\beta-\alpha}}$ so that both $g_a\phi$ and $g\phi$ belong to $\mathbf{A}_{\beta}^2$.  Then from 
\begin{equation} 
\| g(a z)\phi(z) - g(e^{i\theta}z)\phi(z) \|_{\beta}^2 \leq 2\|g_a(z)\phi(z)\|_{\beta}^2+2\|g(e^{i\theta}z)\phi(z)\|_{\beta}^2\leq 4 \|g\|_{\mathcal{B}_{(\beta-\alpha)/2}}^2 \|\phi\|_{\alpha}^2, \nonumber 
\end{equation}
we see that, for any given $\varepsilon>0$, there is a constant $r_0\in (0,1)$ such that
\begin{equation}\label{lemma-1}
{\bf Int_{1}}=\frac{\beta+1}{\pi}\int_{|z|>r_0}|g(a z)\phi(z) - g(e^{i\theta}z)\phi(z)|^{2}(1-|z|^2)^{\beta}dxdy<\varepsilon. 
\end{equation}
We proceed to consider the following integral
\begin{equation}\label{lemma-1.5}
{\bf Int_{2}}=\frac{\beta+1}{\pi}\int_{|z|\leq r_0}|g(a z)\phi(z) - g(e^{i\theta}z)\phi(z)|^{2}(1-|z|^2)^{\beta}dxdy. 
\end{equation}
Note that, on the closed disk $\{ |z| \leq r_0 \}$, $g(az)$ converges uniformly to $g(e^{i\theta}z)$ as $\Delta \ni a \to e^{i\theta}$, i.e.,
$$
\lim_{\Delta \ni a \to e^{i\theta}} \sup_{|z| \leq r_0} |g(az)-g(e^{i\theta}z)| = 0.
$$
Hence for given $\varepsilon>0$, there is a constant $\delta>0$ such that  $$
\sup_{|z| \leq r_0} |g(az)-g(e^{i\theta}z)|(1-|z|^2)^{(\beta-\alpha)/2}<\sqrt{\varepsilon},
$$
for all $a\in \Delta$ with $|a-e^{i\theta}|<\delta$.  Consequently, when $a\in \Delta$ with $|a-e^{i\theta}|<\delta$, we have 
\begin{equation}\label{lemma-2}
{\bf Int_{2}}\leq {\varepsilon} \frac{\beta+1}{\pi} \int_{|z|\leq r_0}|\phi(z)|^{2}(1-|z|^2)^{\alpha}dxdy\leq {\varepsilon} \frac{\beta+1}{(\alpha+1)\pi}\|\phi\|_{\alpha}^2.
\end{equation}
From (\ref{lemma-1}), \ref{lemma-1.5} and (\ref{lemma-2}), we conclude that $$
\lim_{\Delta \ni a \to e^{i\theta}} \| g(a z)\phi(z) - g(e^{i\theta}z)\phi(z) \|_{\beta}^2 = 0,$$
for any $\theta \in (-\pi,\pi]$. This proves Sub-Lemma \ref{sub-2}. 

Finally, for any $\theta \in (-\pi, \pi]$, it is easy to check that $\|\phi(e^{i\theta} z)\|_{\alpha}=\|\phi(z)\|_{\alpha}$ for any $\phi \in \mathbf{A}_{\alpha}^2$. Then we can obtain by a change of variables that $$\mathbb{M}_{g(e^{i\theta}z)}(\alpha, \beta)=\mathbb{M}_g(\alpha, \beta),$$
for any $\theta \in (-\pi, \pi]$. On the other hand, by means of Sub-Lemma \ref{sub-1} and \ref{sub-2}, we know that, for a fixed $\phi \in \mathbf{A}_{\alpha}^2$, the mapping 
$$\ddag: a\mapsto g(az)\phi(z),$$
is analytic from the unit disk $\Delta$ to $\mathbf{A}_{\beta}^2$ and continuous on the unit circle. Then $\|g(az)\phi(z)\|_{\beta}$ is subharmonic in $\Delta$, see \cite{HP}. It follows from the maximum modulus principle that 
$$\mathbb{M}_{g_a}(\alpha, \beta)\leq\mathbb{M}_g(\alpha, \beta),$$
for any $a\in \Delta$. This completes the proof of Lemma \ref{lemma-key}.
\end{proof}

We also need the following well-known weighted Hardy inequality. For completeness, we include a standard proof here for this result. 
\begin{lemma}\label{last}Let $p>1, q>0$, $a<p-1$. Let $f$ be a measurable function on $\mathbb{R}_{+}$ such that $\int_0^\infty x^{a}|f(x)|^p\,dx<\infty$. Then we have
\begin{equation}\label{whardy}
\int_0^\infty x^{a}\Bigl|\frac{f_q(x)}{x^q}\Bigr|^pdx\leq C_{a,p,q}^p\int_0^\infty x^{a}|f(x)|^p\,dx,
\end{equation}
where 
$$f_q(x)=\frac{1}{\Gamma(q)}\int_0^x (x-t)^{q-1}f(t)\,dt, \, \,C_{a,p,q}=\frac{\Gamma\bigl(1-\frac{a+1}{p}\bigr)}{\Gamma\bigl(q+1-\frac{a+1}{p}\bigr)}.
$$
\end{lemma}
\begin{proof}
For fixed $x>0$, let $t = xs$, then $s \in (0,1)$ and
$$
\frac{f_q(x)}{x^q} = \frac{1}{\Gamma(q)} \int_0^1 (1-s)^{q-1} f(xs)ds.
$$

Then, by using Minkowski's inequality, we obtain that
$$
\left(\int_0^\infty x^a \left| \frac{f_q(x)}{x^q} \right|^p dx \right)^{1/p}
\leq \int_0^1 b(s) \left( \int_0^\infty x^a |f(xs)|^p dx \right)^{1/p}ds,
$$
here, $b(s) = {(1-s)^{q-1}}/{\Gamma(q)}$. For fixed $s>0$, let $y = xs$ so that
$$
\int_0^\infty x^a |f(xs)|^pdx = s^{-a-1} \int_0^\infty y^a |f(y)|^p dy.
$$
Consequently, 
\begin{equation}\label{w-1}\left( \int_0^\infty x^a \left| \frac{f_q(x)}{x^q} \right|^p dx \right)^{1/p}
\leq  \left( \int_0^\infty y^a |f(y)|^p dy \right)^{1/p}\int_0^1 b(s) s^{-(a+1)/p}ds.\end{equation}
On the other hand, by the definition of Gamma function, we have 
\begin{equation}\label{w-2}
\int_0^1 b(s) s^{-(a+1)/p}ds=\frac{1}{\Gamma(q)} \int_0^1 (1-s)^{q-1}s^{-(a+1)/p}ds=C_{a,p,q}.\nonumber
\end{equation}
Then the inequality (\ref{whardy}) follows from (\ref{w-1}) and (\ref{w-2}). The lemma is proved.
\end{proof}

\section{\bf {Proof of Theorem \ref{m-1}}}
To prove Theorem \ref{m-1}, we first  recall the notations and definitions from \cite{Jin-c}.  For a number $n\in \mathbb{N}_0=\mathbb{N}\cup \{0\}$, we will use the notation $P_n(z)$ to denote a polynomial of degree at most $n$. We denote by $\mathcal{R}$ the class of all rational functions $R$ in $\Delta$ with $R(0)=R'(0)-1=0$ and which have no poles in $\Delta$. Let $\Omega$ be a subset of the complex plane. If $R$ is analytic in $\Omega$, we say $R$ has {\em no zeros} (resp. one zero) in $\Omega$ if the equation $R(z)=0$ has no solution (resp. exactly one solution) in $\Omega$. If $R$ is analytic in $\Omega \setminus \Omega_1$, where $\Omega_1\subseteq \Omega$ is a finite (possible empty), we will say $R$ has {\em no critical points} (resp. exactly one critical point) in $\Omega$ if there is no (resp. exactly one) $z\in \Omega$ such that $R$ is analytic at $z$ and $R'(z)=0$.

We will consider the following subclasses $\mathcal{Q}_{I}$, $\mathcal{Q}_{II}$, $\mathcal{Q}_{III}$ of the class of all locally univalent functions in $\Delta$ for the purpose of the paper. 

\begin{definition}Let $R$ be a locally univalent function in $\Delta$, and it is analytic on $\mathbb{T}$ except finitely many singular points.

$\bullet$ We say $R$ belongs to the class $\mathcal{Q}_I$, if $R'$ can be written in the following form
 \begin{equation}\label{new-d-0}R'(z)=\frac{[(z-z_1)(z-z_2)\cdots(z-z_s)]P_{\bf n}(z)}{P_{\bf m}(z)}:=\frac{\Pi(z)P_{\bf n}(z)}{P_{\bf m}(z)}, \,z\in \overline{\Delta}.\end{equation}
Here $s\in \mathbb{N}\cup \{0\}$. When $s\geq 1$, $z_j, j=1,2,\cdots, s$ are all distinct critical points on $\mathbb{T}$ of $R$. The numerator and denominator of $R'$ have no common factors except $1$. $P_{\bf n}(z)$ and $P_{\bf m}(z)$ have no zeros in $\overline{\Delta}$. The case $s=0$ means that $\Pi(z)\equiv1$ and $R$ has no critical points on $\mathbb{T}$.

$\bullet$ We say $R$ belongs to the class $\mathcal{Q}_{II}$, if $R'$ can be written in the following form
\begin{eqnarray}\label{new-d-1}
R'(z)=\frac{[(z-z_1)(z-z_2)\cdots(z-z_s)]P_{\bf n}(z)}{\prod\limits_{j=1}^{l}(z-e^{i\theta_j})^2 P_{{\bf m}}(z)}, \,z\in \overline{\Delta}. \end{eqnarray}
Here $s\in \mathbb{N}\cup \{0\}, l\in \mathbb{N}$. When $s\geq 1$, $z_j, j=1,2,\cdots, s$ are all distinct critical points on $\mathbb{T}$ of $R$. The numerator and denominator of $R'$ in (\ref{new-d-1}) have no common factors except $1$. $P_{\bf n}(z)$ and $P_{\bf m}(z)$ have no zeros in $\overline{\Delta}$. In particular, $R'$ has no zeros on $\mathbb{T}$ when $s=0$. 

$\bullet$ We say $R$ belongs to the class $\mathcal{Q}_{III}$, if $R'$ can be written in the following form
\begin{eqnarray}\label{new-d-2}
R'(z)&=& \frac{[(z-z_1)(z-z_2)\cdots(z-z_s)]P_{\bf n}(z)}{\prod\limits_{j=1}^{l}(z-e^{i\theta_j})^3\prod\limits_{k=1}^{t}(z-e^{i\widetilde{\theta}_k})^2P_{\bf m}(z)}, \,z\in \overline{\Delta}.\end{eqnarray}
Here $s, t\in \mathbb{N}\cup \{0\}, l\in \mathbb{N}$. When $s\geq 1$, $z_j, j=1,2,\cdots, s$ are all distinct critical points on $\mathbb{T}$ of $R$. The numerator and denominator of $R'$ in (\ref{new-d-2}) have no common factors except $1$. $P_{\bf n}(z)$ and $P_{\bf m}(z)$ have no zeros in $\overline{\Delta}$. In particular, $R'$ has no zeros on $\mathbb{T}$ when $s=0$ and $R'$ has no 
poles of order $2$ on $\mathbb{T}$ when $t=0$. 
\end{definition}
\begin{remark}
From the arguments in \cite{Jin-c}, we know that $\mathcal{U}_R$ is contained in ${\mathcal{Q}_{I}} \cup  {\mathcal{Q}_{II}} \cup  {\mathcal{Q}_{III}}.$  
\end{remark}

Now, to prove Theorem \ref{m-1}, it is enough to show that the following result holds. 
\begin{proposition}\label{pro}
Let $\alpha>-1$. Let $R$ be locally univalent in $\Delta$ and belong to  $\mathcal{Q}_{I}\cup \mathcal{Q}_{II} \cup \mathcal{Q}_{III}$. If $R$ has at least one critical point on $\mathbb{T}$, then $S(R)$ is a multiplier from ${\bf A}_{\alpha}^2$ to ${\bf A}_{\alpha+4}^2$ and
 $$\|S(R)\|_{{\bf M}_{\alpha, \alpha+4}}\geq \mathbb{M}_{\kappa}(\alpha).$$
\end{proposition}

\begin{proof}[Proof of Proposition \ref{pro}] Assume that $R$ has at least one critical point on $\mathbb{T}$. We first show that $S(R)$ is a multiplier from ${\bf A}_{\alpha}^2$ to ${\bf A}_{\alpha+4}^2$. 

When $R$ belongs to $\mathcal{Q}_{I}$, as in (\ref{new-d-0}), we have 
$$R'(z)=\frac{[(z-z_1)(z-z_2)\cdots(z-z_s)]P_{\bf n}(z)}{P_{\bf m}(z)}, s\geq 1.$$
A direct computation gives that
$$N(R)(z)=\sum_{m=1}^{s}\frac{1}{z-z_m}+\frac{P'_{\bf n}(z)}{P_{\bf n}(z)}-\frac{P'_{\bf m}(z)}{P_{\bf m}(z)}.$$
It follows that
$$[N(R)(z)]'=-\sum_{m=1}^{s}\frac{1}{(z-z_m)^2}+\Big[\frac{P'_{\bf n}(z)}{P_{\bf n}(z)}\Big]'-\Big[\frac{P'_{\bf m}(z)}{P_{\bf m}(z)}\Big]'.$$
Then, from the definition of Schwarzian derivative   
$$S(R)(z)=[N(R)(z)]'-\frac{1}{2}[N(R)(z)]^2,$$ 
and by using the element identity
$$(\sum_{m=1}^{n}a_m)^2=\sum_{m=1}^n {a_m^2}+2\sum_{i=1}^{n-1}\sum_{j=i+1}^n a_{i}a_{j}, \,\, n\geq 2.$$
we conclude that $S(R)$ has the following form
\begin{equation}\label{s-1}
S(R)(z)=-\frac{3}{2}\sum_{m=1}^{s}\frac{1}{(z-z_m)^2}+ \sum_{m=1}^{s}\frac{A_m(z)}{z-z_m}+{\bf Anal}(z).
\end{equation}
Here,  
\begin{equation}\label{anal}
{\bf Anal}(z)=\Big[\frac{P'_{\bf n}(z)}{P_{\bf n}(z)}\Big]'-\Big[\frac{P'_{\bf m}(z)}{P_{\bf m}(z)}\Big]'-\frac{1}{2}\Big[\frac{P'_{\bf n}(z)}{P_{\bf n}(z)}\Big]^2-\frac{1}{2}\Big[\frac{P'_{\bf m}(z)}{P_{\bf m}(z)}\Big]^2,\end{equation}
For each $r\in [1, s]$, $A_r(z)$ is analytic in $\overline{\Delta}$. 

When $R$ belongs to $\mathcal{Q}_{II}$, as in (\ref{new-d-1}), we have 
$$R'(z)=\frac{[(z-z_1)(z-z_2)\cdots(z-z_s)]P_{\bf n}(z)}{\prod\limits_{j=1}^{l}(z-e^{i\theta_j})^2 P_{{\bf m}}(z)}, l\geq 1, s\geq 1.$$
Then we have
$$N(R)(z)=\sum_{m=1}^{s}\frac{1}{z-z_m}-2\sum_{j=1}^{l} \frac{1}{z-e^{i\theta_j}}+\frac{P'_{\bf n}(z)}{P_{\bf n}(z)}-\frac{P'_{\bf m}(z)}{P_{\bf m}(z)},$$
so that 
$$[N(R)(z)]'=-\sum_{m=1}^{s}\frac{1}{(z-z_m)^2}+2\sum_{j=1}^{l} \frac{1}{(z-e^{i\theta_j})^2}+\Big[\frac{P'_{\bf n}(z)}{P_{\bf n}(z)}\Big]'-\Big[\frac{P'_{\bf m}(z)}{P_{\bf m}(z)}\Big]'.$$
Then, similarly as above, we have
\begin{eqnarray}\label{s-2}
\lefteqn{S(R)(z)=-\frac{3}{2}\sum_{m=1}^{s}\frac{1}{(z-z_m)^2}+ \sum_{m=1}^{s}\frac{B_m(z)}{z-z_m}}
\\
&&\quad\quad\,\,\,\,+\sum_{j=1}^{l} \frac{C_j(z)}{z-e^{i\theta_j}}+{\bf Anal}(z).\nonumber 
\end{eqnarray}
Here, ${\bf Anal}(z)$ is defined as in (\ref{anal}). for each $m\in [1, s]$, $j\in [1, l]$, $B_m(z)$ and $C_j(z)$ are analytic in $\overline{\Delta}$. 

When $R$ belongs to $\mathcal{Q}_{III}$, as in (\ref{new-d-2}), we have 
$$R'(z)=\frac{[(z-z_1)(z-z_2)\cdots(z-z_s)]P_{\bf n}(z)}{\prod\limits_{j=1}^{l}(z-e^{i\theta_j})^3\prod\limits_{k=1}^{t}(z-e^{i\widetilde{\theta}_k})^2P_{\bf m}(z)},\, s\geq 1, l\geq 1, t\geq 0.$$
Then we have 
$$N(R)(z)=\sum_{m=1}^{s}\frac{1}{z-z_m}-3\sum_{j=1}^{l} \frac{1}{z-e^{i\theta_j}}-2\sum_{k=1}^{t}\frac{1}{z-e^{i\widetilde{\theta}_k}}+\frac{P'_{\bf n}(z)}{P_{\bf n}(z)}-\frac{P'_{\bf m}(z)}{P_{\bf m}(z)},$$
and 
\begin{eqnarray}\lefteqn{[N(R)(z)]'=-\sum_{m=1}^{s}\frac{1}{(z-z_m)^2}+3\sum_{j=1}^{l} \frac{1}{(z-e^{i\theta_j})^2}}\nonumber
\\
&&\quad\quad\quad\quad+2\sum_{k=1}^{t} \frac{1}{(z-e^{i\widetilde{\theta}_k})^2}+\Big[\frac{P'_{\bf n}(z)}{P_{\bf n}(z)}\Big]'-\Big[\frac{P'_{\bf m}(z)}{P_{\bf m}(z)}\Big]'.\nonumber\end{eqnarray}
Consequently, similarly again as above, we have, for $t\geq 1$,
\begin{eqnarray}\label{s-3}
\lefteqn{S(R)(z)=-\frac{3}{2}\sum_{m=1}^{s}\frac{1}{(z-z_m)^2}+ \sum_{m=1}^{s}\frac{D_m(z)}{z-z_m}} 
\\
&&\quad\quad\,\,\,\,-\frac{3}{2}\sum_{j=1}^{l}\frac{1}{(z-e^{i\theta_j})^2}+\sum_{j=1}^{l} \frac{E_j(z)}{z-e^{i\theta_j}}\nonumber \\
&&\quad\quad\,\,\,\, +\sum_{k=1}^{t} \frac{F_k(z)}{z-e^{i\widetilde{\theta}_k}}+{\bf Anal}(z).\nonumber
\end{eqnarray}
Here,  ${\bf Anal}(z)$ is the same as in (\ref{anal}) and $D_m, E_j, F_k$ are all analytic in $\Delta$ for each $m\in [1, s], j\in [1,l], k\in [1, t]$. Note that when $t=0$, there are no following terms
$$\sum_{k=1}^{t} \frac{F_k(z)}{z-e^{i\widetilde{\theta}_k}},$$
in the right side of (\ref{s-3}). 
Combining (\ref{s-1}), (\ref{s-2}) and (\ref{s-3}), we conclude that $S(R)$ belongs to $\mathcal{B}_2$ so that it is a multiplier from ${\bf A}_{\alpha}^2$ to ${\bf A}_{\alpha+4}^2$.  

To proceed the proof, we need the following claim.
\begin{claim}\label{claim-1}
Let $\alpha>-1$ and let $\alpha+2<2\lambda<\alpha+3$. Let $A(z)$ be analytic in $\overline{\Delta}$. 

{\bf (1)} We define 
 \begin{equation}\label{c-01}
\mathbf{T}_{1}:=\frac{\alpha+5}{\pi}\iint_{\Delta}|A(z)||1-rz|^{-2\lambda-3}
(1-|z|^2)^{\alpha+4}dxdy,\nonumber
 \end{equation}   
and  
 \begin{equation}\label{c-02}
\mathbf{T}_{2}:=\frac{\alpha+5}{\pi}\iint_{\Delta}|A(z)||1-rz|^{-2\lambda-2}
(1-|z|^2)^{\alpha+4}dxdy,\nonumber \end{equation} 
Then there is a constant $M>0$ such that ${\bf T}_{i}\leq M$ for all $i=1,2$ and all $r\in (0,1)$.

{\bf (2)} Let $\Theta\in (-2\pi,2\pi)$ be such that $e^{-i\Theta}\neq 1$. We define  
 \begin{equation}\label{c-03}
\mathbf{T}_{3}:=\frac{\alpha+5}{\pi}\iint_{\Delta}|A(z)||1-re^{-i\Theta}z|^{-1}|1-rz|^{-2\lambda-2}
(1-|z|^2)^{\alpha+4}dxdy,\nonumber
 \end{equation}   
and  
 \begin{equation}\label{c-04}
\mathbf{T}_{4}:=\frac{\alpha+5}{\pi}\iint_{\Delta}|A(z)||1-re^{-i\Theta}z|^{-2}|1-rz|^{-2\lambda-2}
(1-|z|^2)^{\alpha+4}dxdy,\nonumber \end{equation} 
Then there is a constant $M>0$ such that ${\bf T}_{i}\leq M$ for all $i=3,4$ and all $r\in (0,1)$.
\end{claim}
\begin{proof}\label{claim}First, note that there is a constant $M_1>0$ such that $|A(z)|\leq M_1$ for all $z\in \Delta$ since $A(z)$ is analytic in $\overline{\Delta}$. Then we have
$${\bf T}_1\leq M_1 \|(1-rz)^{-\lambda-\frac{3}{2}}\|_{\alpha+4}^2, \,\,{\bf T}_2\leq M_1 \|(1-rz)^{-\lambda-1}\|_{\alpha+4}^2.$$
Consequently, by using the third case of Lemma \ref{l-2} and its proof, since $2\lambda+2<2\lambda+3<(\alpha+4)+2$, we see that there is a constant $M>0$ such that ${\bf T}_{i}\leq M$ for all $i=1,2$ and all $r\in (0,1)$. This proves the case {\bf (1)}.

We next only prove the case for ${\bf T}_{4}$, The case for ${\bf T}_{3}$ can be similarly proved.
In view of the conditions of the case {\bf (2)}, for some fixed $r_0\in (0,1)$, we can take two constants $\eta_1>0, \eta_2>0$ such that $\overline{\Omega}_1 \cap \overline{\Omega}_2=\emptyset$. 
Here $$\Omega_1=\{z=re^{i\theta}: r_0<r<1, \Theta-\eta_1<\theta<\Theta+\eta_1\}, $$
$$\Omega_2=\{z=re^{i\theta}: r_0<r<1, -\eta_2<\theta<\eta_2\}.$$
We set 
$$\phi_r(z)=\frac{\alpha+5}{\pi}|1-re^{-i\Theta}z|^{-2}|1-rz|^{-2\lambda-2}(1-|z|^2)^{\alpha+4}.$$
Since $|A(z)|$ is bounded in $\Delta$, it is enough to prove that $\iint_{\Delta}\phi_r(z)dxdy<M$ for some $M>0$.
Note that 
$$\sup_{r\in (0,1), z\in \Omega_1}|1-rz|^{-2\lambda-2}=\max_{z\in {\pi_1}}\{|1-z|^{-2\lambda-2}\},$$
here, $\pi_1$ is the sector $\{z=re^{i\theta}: r\in [0,1], \theta\in [\Theta-\eta_1, \Theta+\eta_1]\}$. Since $1-z$ is continuous in the compact set ${\bf \pi}_1$ and the point $1$ does not belong to ${\bf \pi}_1$, then we see that there is a constant $C_1>0$ such that
$$|1-rz|^{-2\lambda-2}\leq C_1,$$
for all $z\in \Omega_1$ and $r\in (0,1)$. Similarly, since 
$$\sup_{r\in (0,1), z\in \Omega_2}|1-re^{i\Theta}z|^{-2}=\sup_{r\in (0,1), z\in \Omega_2}|e^{-i\Theta}-rz|^{-2}=\max_{z\in {\pi_2}}\{|e^{-i\Theta}-z|^{-2}\},$$
here, $\pi_2$ is the sector $\{z=re^{i\theta}: r\in [0,1], \theta\in [-\eta_2, \eta_2]\}$. As $e^{-i\Theta}-z$ is continuous in the compact set $\pi_2$ and the point $e^{-i\Theta}$ is not contained in $\pi_2$, hence we know that there is a constant $C_2>0$ such that
$$|1-re^{i\Theta}z|^{-2}\leq C_2,$$
for all $z\in \Omega_2$ and $r\in (0,1)$. Meanwhile, we have
$$\sup_{r\in (0,1), z\in \Delta-\Omega_1-\Omega_2}|1-rz|^{-2\lambda-2}|1-re^{i\Theta}z|^{-2}=\max_{z\in \pi_3}\{|1-z|^{-2\lambda-2}|e^{-i\Theta}-z|^{-2}\}$$
here, $\pi_3$ is the compact set $\overline{\Delta}-(s_1 \cup s_2)$, where
$$s_1:=\{z=re^{i\theta}: 0<r\leq 1, \Theta-\eta_1<\theta<\Theta+\eta_1\},$$
and $$s_2:=\{z=re^{i\theta}: 0<r\leq 1, -\eta_2<\theta<\eta_2\}.$$
Therefore, from the facts that both the points $1$ and $e^{-i\Theta}$ are not contained in $\pi_3$, and $1-z, e^{-i\Theta}-z$ are continuous in $\pi_3$, we see that there is a constant $C_3>0$ such that
$$|1-re^{i\Theta}z|^{-2}|1-rz|^{-2\lambda-2}\leq C_3,$$
for all $z\in \Delta-\Omega_1-\Omega_2$ and $r\in (0,1)$. 
Consequently, we obtain that 
\begin{eqnarray}\label{phi-1}
\iint_{\Omega_1}\phi_r(z)dxdy& \leq & C_1\frac{\alpha+5}{\pi}\iint_{\Omega_1}|1-re^{-i\Theta}z|^{-2}(1-|z|^2)^{\alpha+4}dxdy  \\
& \leq& C_1\|(1-rz)^{-1}\|_{\alpha+4}^2,\nonumber
 \end{eqnarray} 
 and 
\begin{eqnarray}\label{phi-2}
\iint_{\Omega_2}\phi_r(z)dxdy&\leq& C_2\frac{\alpha+5}{\pi}\iint_{\Omega_2}|1-rz|^{-2\lambda-2}(1-|z|^2)^{\alpha+4}dxdy   \\
& \leq& C_2\|(1-rz)^{-\lambda-1}\|_{\alpha+4}^2,\nonumber
 \end{eqnarray}and  
 \begin{eqnarray}\label{phi-3} 
 \iint_{\Delta-\Omega_1-\Omega_2}\phi_r(z)dxdy & \leq& C_3\frac{\alpha+5}{\pi}\iint_{\Delta-\Omega_1-\Omega_2}(1-|z|^2)^{\alpha+4}dxdy  \\
& \leq& C_3\|1\|_{\alpha+4}^2=C_3.\nonumber
 \end{eqnarray} 
By (\ref{phi-1}), (\ref{phi-2}), (\ref{phi-3}), since $1<(\alpha+4)+2$, $2\lambda+2<(\alpha+3)+2<(\alpha+4)+2$, using again the third case of Lemma \ref{l-2}, we conclude that for some constant $M>0$ it holds that 
\begin{equation*} 
\iint_{\Delta}\phi_r(z)dxdy=\Big(\iint_{\Omega_1}+\iint_{\Omega_2}+\iint_{\Delta-\Omega_1-\Omega_2}\Big)\phi_r(z)dxdy\leq M, \nonumber
 \end{equation*}
 for all $r\in (0,1)$. The claim is proved. 
  \end{proof}
  
We continue with the proof of Proposition \ref{pro}.  We only prove the case when $R$ belongs to the class $\mathcal{Q}_{III}$. 
As will be seen below, all other cases are covered in the following proof. For a complex number $z$, we next will use $\arg z$ to denote an argument of $z$ with $\arg z \in (-\pi, \pi]$. 

When $R$ belongs to the class $\mathcal{Q}_{III}$, let $\alpha+2<2\lambda<\alpha+3$. From Lemma \ref{l-2}, we get that
\begin{eqnarray}\label{jia-1}
\mathbb{M}_{R}^2(\alpha) &\geq& \sup\limits_{r\in (0,1)}\|S(R)(re^{i\arg z_1}z)\|_{{\bf M}_{\alpha, \alpha+4}}^2 \nonumber \\
&\geq &\sup\limits_{r\in (0,1)}\frac{\|S(R)(re^{i\arg z_1}z)(1-rz)^{-\lambda}\|_{\alpha+4}^2}{\|(1-rz)^{-\lambda}\|_{\alpha}^2}
\end{eqnarray}
We establish an estimate for $\|S(R)(re^{i\arg z_1}z)(1-rz)^{-\lambda}\|_{\alpha+4}^2$.

We write ${\bf\Pi}(z):=(1-rz)^{-\lambda}$ and obtain from (\ref{s-3}) that 
\begin{eqnarray}\label{key-in}
\lefteqn{S(R)(re^{i\arg z_1}z){\bf\Pi}(z)=-\frac{3}{2}\frac{z_1^{-2}{\bf\Pi}(z)}{(1-rz)^2}-\frac{{z_1}^{-1}D_1(z){\bf\Pi}(z)}{1-rz}} \\
&& -\frac{3}{2}\sum_{m=2}^{s}\frac{z_m^{-2}{\bf\Pi}(z)}{(1-re^{i(\arg z_1-\arg z_m)}z)^2}+\sum_{m=2}^{s}\frac{-z_m^{-1}D_m(z){\bf\Pi}(z)}{1-re^{i(\arg z_1-\arg z_m)}z}\nonumber
\\
&&-\frac{3}{2}\sum_{j=1}^{l}\frac{e^{-2i\theta_j}{\bf\Pi}(z)}{(1-re^{i(\arg z_1-\theta_j)}z)^2}+\sum_{j=1}^{l}\frac{-e^{-i\theta_j}E_j(z){\bf\Pi}(z)}{1-re^{i(\arg z_1-\theta_j)}z}\nonumber \\
&&+\sum_{k=1}^{t} \frac{-e^{-i\widetilde{\theta}_k}F_k(z){\bf\Pi}(z)}{1-re^{i(\arg z_1-\widetilde{\theta}_k)}z}+{\bf Anal}(re^{i\arg z_1}z){\bf\Pi}(z).\nonumber
\end{eqnarray}
Here, note that $\arg z_1-\arg z_m, \arg z_1-\theta_j, \arg z_1-\widetilde{\theta}_k$ are pairwise distinct for all $m,j,k$. We regard the right side of (\ref{key-in}) as a sum of $(2s+2l+t+1)$ terms listed above. Consequently, in the integral $\|S(R)(re^{i\arg z_1}z)(1-rz)^{-\lambda}\|_{\alpha+4}^2$, we first use the following 
$$\Big|\sum_{m=1}^{n}a_m\Big|^2\geq |a_1|^2-2|a_1|\sum_{m=2}^{n}|a_m|,$$
with respect to (\ref{key-in}) by taking $n=2s+2l+t+1$ and 
$$a_1=-\frac{3}{2}\frac{z_1^{-2}{\bf\Pi}(z)}{(1-rz)^2},$$
and then integrate both sides of the inequality to obtain that
\begin{eqnarray}
\|S(R)(re^{i\arg z_1}z)(1-rz)^{-\lambda}\|_{\alpha+4}^2\geq \|-\frac{3}{2}\frac{z_1^{-2}{\bf\Pi}(z)}{(1-rz)^2}\|_{\alpha+4}^2-2\bf \Sigma. \nonumber
\end{eqnarray}
Here, $\bf \Sigma$ denotes a sum of $2s+2l+t$ integrals, each of which is of one of the types ${\bf T}_{1}, {\bf T}_{2}, {\bf T}_{3}, {\bf T}_{4}$ appearing in Claim \ref{claim-1}. Therefore, from Claim \ref{claim-1}, we see that there is a constant ${\bf M}>0$ such that
\begin{eqnarray}
\|S(R)(re^{i\arg z_1}z)(1-rz)^{-\lambda}\|_{\alpha+4}^2\geq \|-\frac{3}{2}\frac{z_1^{-2}{\bf\Pi}(z)}{(1-rz)^2}\|_{\alpha+4}^2-{\bf M}, \nonumber
\end{eqnarray}
for all $r\in (0,1)$. It follows from (\ref{jia-1}) that
\begin{eqnarray}\label{jia-2}
\mathbb{M}_{R}^2(\alpha) &\geq& \sup\limits_{r\in (0,1)}\frac{\|-\frac{3}{2}\frac{z_1^{-2}{\bf\Pi}(z)}{(1-rz)^2}\|_{\alpha+4}^2-{\bf M}}{\|(1-rz)^{-\lambda}\|_{\alpha}^2} \nonumber \\
&\geq &\sup\limits_{r\in (0,1)}\Big(\frac{9}{4}\frac{\|(1-rz)^{-2\lambda-2}\|_{\alpha+4}^2}{\|(1-rz)^{-\lambda}\|_{\alpha}^2}-\frac{{\bf M}}{\|(1-rz)^{-\lambda}\|_{\alpha}^2}\Big).
\end{eqnarray}
On the other hand, when $\alpha+2<2\lambda<\alpha+3$, by Lemma \ref{l-2}, we have
\begin{equation}\label{jia-3}
 \|(1-rz)^{-\lambda}\|_{\alpha}^2=\frac{\Gamma(\alpha+2)\Gamma(2\lambda-\alpha-2)}{[\Gamma(\lambda)]^2}\cdot\frac{1+{o}(1)}{(1-r^2)^{2\lambda-\alpha-2}},\, {\textup{as}}\,\, r\rightarrow 1^{-}, 
\end{equation} 
and 
\begin{equation}\label{jia-4}
 \|(1-rz)^{-\lambda-2}\|_{\alpha+4}^2=\frac{\Gamma(\alpha+6)\Gamma(2\lambda-\alpha-2)}{[\Gamma(\lambda+2)]^2}\cdot\frac{1+{o}(1)}{(1-r^2)^{2\lambda-\alpha-2}},\, {\textup{as}}\,\, r\rightarrow 1^{-}. 
\end{equation} 
Combining (\ref{jia-2}), (\ref{jia-3}) and (\ref{jia-4}), let $r\to 1^{-}$, we obtain that
\begin{equation}
\mathbb{M}_{R}^2(\alpha)  \geq\frac{9}{4}\frac{\Gamma(\alpha+6)}{\Gamma(\alpha+2)}\frac{[\Gamma(\lambda)]^2}{[\Gamma(\lambda+2)]^2}=\frac{9}{4}\frac{(\alpha+5)(\alpha+4)(\alpha+3)(\alpha+2)}{[\lambda(\lambda+1)]^2}.\nonumber 
\end{equation}
Finally, let $\lambda\to (\frac{\alpha}{2}+1)^{+}$, we have 
\begin{equation*}
\mathbb{M}_{R}^2(\alpha) \geq\frac{9}{4}\frac{(\alpha+5)(\alpha+4)(\alpha+3)(\alpha+2)}{[(\frac{\alpha}{2}+1)(\frac{\alpha}{2}+2)]^2}=\frac{36(\alpha+5)(\alpha+3)}{(\alpha+2)(\alpha+4)}=\mathbb{M}_{\kappa}^2(\alpha).
\end{equation*}  
This completes the proof of Proposition \ref{pro} and hence of Theorem \ref{m-1}.
\end{proof}

\section{\bf {Proof of Theorem \ref{m-2}}}
The main ideas of the proof are from the paper \cite{SS}.
We first prove that \begin{equation}\label{geq}\mathbb{M}_{g_0}^2(\alpha, \beta)\geq \frac{1}{2^{\beta-\alpha}}\frac{\Gamma(\beta+2)}{\Gamma(\alpha+2)}\Big[\frac{\Gamma(1+\alpha/2)}{\Gamma(\beta/2+1)}\Big]^2.\end{equation}
From Lemma \ref{l-2}, for $r\in (0,1)$, $2\lambda> \alpha+2$, we know that 
\begin{equation} 
 \|{(1-rz^2)^{-\lambda}}\|_{\alpha}^2=\frac{\Gamma(\alpha+2)\Gamma(2\lambda-\alpha-2)}{2^{\alpha+1}[\Gamma(\lambda)]^2}\cdot\frac{1+{o}(1)}{(1-r^2)^{2\lambda-\alpha-2}},\, {\textup{as}}\,\, r\rightarrow 1^{-}, \nonumber 
 \end{equation}
 and \begin{equation} 
 \|{(1-rz^2)^{-\lambda-\frac{\beta-\alpha}{2}}}\|_{\beta}^2=\frac{\Gamma(\beta+2)\Gamma(2\lambda-\alpha-2)}{2^{\beta+1}[\Gamma(\lambda+\frac{\beta-\alpha}{2})]^2}\cdot\frac{1+{o}(1)}{(1-r^2)^{2\lambda-\alpha-2}},\, {\textup{as}}\,\, r\rightarrow 1^{-}.\nonumber 
 \end{equation}
It follows from Lemma \ref{lemma-key} that  
\begin{eqnarray}
\mathbb{M}_{g_0}^2(\alpha, \beta)&\geq&\mathbb{M}_{g_0(\sqrt{r}z)}^2(\alpha, \beta)\geq \sup\limits_{r\in (0,1)}\frac{ \|{(1-rz^2)^{-\lambda-\frac{\beta-\alpha}{2}}}\|_{\beta}^2}{ \|{(1-rz^2)^{-\lambda}}\|_{\alpha}^2}\nonumber \\
&\geq& \frac{1}{2^{\beta-\alpha}}\frac{\Gamma(\beta+2)}{\Gamma(\alpha+2)}\frac{[\Gamma(\lambda)]^2}{[\Gamma(\lambda+\frac{\beta-\alpha}{2})]^2}. \nonumber\end{eqnarray} 
Then, let $\lambda\to (\frac{\alpha}{2}+1)^{+}$, we obtain (\ref{geq}).

Next, we prove 
$$\mathbb{M}_{g_0}^2(\alpha, \beta)\leq \frac{1}{2^{\beta-\alpha}}\frac{\Gamma(\beta+2)}{\Gamma(\alpha+2)}\Big[\frac{\Gamma(1+\alpha/2)}{\Gamma(\beta/2+1)}\Big]^2.$$
That is to say, for any $\phi\in \mathbf{A}_{\alpha}^2$, we need show that
\begin{equation}\label{eq-d-0}
\|g_0\phi\|_{\beta}^2\leq \frac{1}{2^{\beta-\alpha}}\frac{\Gamma(\beta+2)}{\Gamma(\alpha+2)}\Big[\frac{\Gamma(1+\alpha/2)}{\Gamma(1+\beta/2)}\Big]^2\|\phi\|_{\alpha}^2. 
\end{equation}
That is 
\begin{eqnarray}\label{eq-d}
\lefteqn{\iint_{\Delta}(1-|z|^2)^{\beta}|1-z|^{\beta-\alpha}|\phi(z)|^2dxdy}\nonumber \\ &&\leq \frac{1}{2^{\beta-\alpha}} \frac{\Gamma(\beta+1)}{\Gamma(\alpha+1)}\Big[\frac{\Gamma(1+\alpha/2)}{\Gamma(1+\beta/2)}\Big]^2\iint_{\Delta}|\phi(z)|^2(1-|z|^2)^{\alpha}dxdy.
\end{eqnarray}
We will now establish an equivalent formulation of (\ref{eq-d-0})(also (\ref{eq-d})). By using the Cayley transform
$$z=z(w)=\frac{w-i}{w+i}, \,\,z\in \Delta,\,\, w=u+iv\in \mathbb{C}_{+},$$
which maps conformally $\mathbb{C}_{+}$ onto $\Delta$, we get that
$$1-z=\frac{2i}{w+i},\,\, 1-|z|^2=\frac{4v}{|w+i|^2},\,\, dxdy=\frac{4}{|w+i|^4}dudv.$$
We set 
$$h(w):=\frac{\phi(z(w))}{(w+i)^{\alpha+2}}.$$ A direct computation shows that the left side of  (\ref{eq-d}) equals
$$4^{\alpha+1}\iint_{\mathbb{C}_{+}}\frac{|h(w)|^2}{|w|^{\beta-\alpha}}v^{\beta}dudv,$$
and the integral in the right side of (\ref{eq-d}) equals
$$4^{\alpha+1}\iint_{\mathbb{C}_{+}}|h(w)|^2v^{\alpha}dudv.$$
It follows that (\ref{eq-d-0}) is equivalent to 
\begin{equation}\label{equi}
\iint_{\mathbb{C}_{+}}\frac{|h(w)|^2}{|w|^{\beta-\alpha}}v^{\beta}dudv \leq \frac{1}{2^{\beta-\alpha}} \frac{\Gamma(\beta+1)}{\Gamma(\alpha+1)}\Big[\frac{\Gamma(1+\alpha/2)}{\Gamma(1+\beta/2)}\Big]^2
\iint_{\mathbb{C}_{+}}|h(w)|^2v^{\alpha}dudv.
\end{equation}
From Lemma \ref{l-1}, we know that $h\in \mathbf{A}_{\alpha}^2(\mathbb{C}_{+})$ so that there is a $\widetilde{h}$ that belongs to $L_{\alpha}^2(\mathbb{R}_{+})$ such that
$$
h(w) = \int_{0}^{\infty} \tilde{h}(t) e^{i w t} \, dt, \quad w \in \mathbb{C}_{+},
$$
and 
$$
\|h\|_{\alpha}^{2} = \frac{\Gamma(\alpha+1)}{2^{\alpha}} \int_{0}^{\infty} |\tilde{h}(t)|^{2} \frac{dt}{t^{\alpha+1}}.
$$
Note that for any $v>0$, by the Cauchy-Schwarz inequality, we have 
\begin{eqnarray}\Big[\int_{0}^{\infty}|\widetilde{h}(t)|e^{-vt}dt\Big]^2&=&\Big[\int_{0}^{\infty}e^{-vt}t^{(\alpha+1)/2\cdot|\widetilde{h}(t)|\frac{1}{t^{(\alpha+1)/2}}}dt\Big]^2\nonumber\\
 &\leq&\Big[\int_{0}^{\infty}e^{-2vt}t^{\alpha+1}dt\Big]\Big[\int_{0}^{\infty}|\widetilde{h}(t)|^2 \frac{dt}{t^{\alpha+1}}\Big]\nonumber \\
&=& \frac{\Gamma(\alpha+2)}{(2v)^{\alpha+2}}\|\widetilde{h}\|_{L_{\alpha}^2}^2<\infty.\nonumber \end{eqnarray}
Hence, we can define the Laplace transform of $\widetilde{h}$ by
$$\mathfrak{L}(\widetilde{h})(s):=\int_{0}^{\infty}\widetilde{h}(t)e^{-st}dt,\,\, \Re(s)>0.$$
Consequently, $h(w)=\mathfrak{L}(\widetilde{h})(-iw).$ Now set $\gamma=(\beta-\alpha)/2$ and define
$$
H(t) := \frac{1}{\Gamma(\gamma)} \int_0^t (t-x)^{\gamma-1} \widetilde{h}(x) \, dx,\,\, t>0,
$$
which is a convolution and can be written as  
$$
H(t) =\frac{1}{\Gamma(\gamma)}  (p_{\gamma} *\widetilde{h})(t),\,\, p_{\gamma}(t) := {t}^{\gamma-1}, t>0.
$$
It is well known that for $\Re(s) > 0$, $
\mathfrak{L}\{p_{\gamma}\}(s)=\Gamma(\gamma)\,s^{-\gamma}.$ Here, $s^{-\gamma} = e^{-\gamma\log s}$ and $\ln s$ is a branch of the logarithmic function with $\log 1=0$. 
It follows from the convolution theorem of Laplace transform that
 $$
\mathfrak{L}(H)(s) = \mathfrak{L}(p_{\gamma})(s) \cdot \mathfrak{L}(\widetilde{h})(s)
            = s^{-\gamma} \, \mathfrak{L}(\widetilde{h})(s),\,\,\Re(s)>0.
$$
Then, let $s = -i w$ for $w \in \mathbb{C}_{+}$, since $\Re(s)= \Im(w) > 0$, we can define $$
{\bf{H}}(w) := \int_0^\infty H(t) e^{i w t} \, dt.          
$$ and obtain that ${\bf{H}}(w)=\mathfrak{L}(H)(-i w)$. 
Consequently, we have
$$
{\bf H}(w) =(-i w)^{-\gamma} \, \mathfrak{L}(\widetilde{h})(-i w)=(-i w)^{-\gamma} h(w), \,\,w \in \mathbb{C}_{+}.
$$
It follows that $$|{\bf H}(w)| =\frac{|h(w)|}{|w|^{\frac{\beta-\alpha}{2}}}, \,\, w \in \mathbb{C}_{+},$$
and $\bf H$ belongs to ${\bf A}_{\beta}^2$. Hence we see from Lemma \ref{l-1} that
$$
\|{\bf H}\|_{\beta}^{2} = \frac{\Gamma(\beta+1)}{2^{\beta}} \int_{0}^{\infty} |H(t)|^{2} \frac{dt}{t^{\beta+1}}.
$$
Therefore, we see that the inequality (\ref{equi}) is equivalent to the following
\begin{eqnarray}
\lefteqn{\frac{\Gamma(\beta+1)}{2^{\beta}} \int_{0}^{\infty} |H(t)|^{2} \frac{dt}{t^{\beta+1}}}\nonumber \\&&\leq \frac{1}{2^{\beta-\alpha}} \frac{\Gamma(\beta+1)}{\Gamma(\alpha+1)}\Big[\frac{\Gamma(1+\alpha/2)}{\Gamma(1+\beta/2)}\Big]^2 \frac{\Gamma(\alpha+1)}{2^{\alpha}} \int_{0}^{\infty} |\tilde{h}(t)|^{2} \frac{dt}{t^{\alpha+1}}.\nonumber
\end{eqnarray}
Furthermore, (\ref{eq-d-0}) is equivalent to 
\begin{eqnarray}\label{equi-1}
\lefteqn{\int_{0}^{\infty} \Big|\frac{1}{\Gamma((\beta-\alpha)/2)} \int_0^t (t-x)^{\frac{\beta-\alpha}{2}-1} \widetilde{h}(x)dx\Big|^{2} \frac{dt}{t^{\beta+1}}}\nonumber \\&&\leq \Big[\frac{\Gamma(1+\alpha/2)}{\Gamma(1+\beta/2)}\Big]^2 \int_{0}^{\infty} |\tilde{h}(t)|^{2} \frac{dt}{t^{\alpha+1}}.
\end{eqnarray}

Finally, we recall the weighted Hardy inequality in Lemma \ref{last}. Then, the inequality (\ref{equi-1}) follows by taking $q=\frac{\beta-\alpha}{2}$, $a=-\alpha-1$, and $p=2$ in (\ref{whardy}). This proves (\ref{eq-d-0}) is true.  The proof of Theorem \ref{m-2} is finished. 
\section{{\bf Final remarks}}
\begin{remark} 
For $\beta>\alpha>-1$, take $$g_1(z)=2^{(\alpha-\beta)/2} (1-z)^{-\frac{\beta-\alpha}{2}}, \,z\in \Delta.$$
We can check that $g_1\in \mathcal{B}_{(\beta-\alpha)/2}$ and $\|g_1\|_{\mathcal{B}_{(\beta-\alpha)/2}}=1.$ We will show that 
\begin{equation}\label{remark-0}
{\mathbb{M}}_{g_1}(\alpha,\beta)=\mathbb{M}_{g_0}(\alpha,\beta).
\end{equation}
First note that $\mathbb{M}_{g_1}(\alpha,\beta)\leq\mathbb{M}_{g_0}(\alpha,\beta)$ because for any $z\in \Delta,$ 
$$|g_{1}(z)|/|g_0(z)|\leq \frac{|(1+z)^{\frac{\beta-\alpha}{2}}|}{2^{(\beta-\alpha)/2}}\leq 1.$$
Next we will show that
\begin{equation}\label{remark-1}\mathbb{M}_{g_1}^2(\alpha,\beta)\geq \mathbb{M}_{g_0}^2(\alpha,\beta)=\frac{1}{2^{\beta-\alpha}}\frac{\Gamma(\beta+2)}{\Gamma(\alpha+2)}\Big[\frac{\Gamma(1+\alpha/2)}{\Gamma(1+\beta/2)}\Big]^2.\end{equation}
From Lemma \ref{l-2}, for $r\in (0,1)$, $2\lambda> \alpha+2$, we have 
\begin{equation} 
 \|{(1-rz)^{-\lambda}}\|_{\alpha}^2=\frac{\Gamma(\alpha+2)\Gamma(2\lambda-\alpha-2)}{[\Gamma(\lambda)]^2}\cdot\frac{1+{o}(1)}{(1-r^2)^{2\lambda-\alpha-2}},\, {\textup{as}}\,\, r\rightarrow 1^{-}, \nonumber 
 \end{equation}
 and \begin{equation} 
 \|{(1-rz)^{-\lambda-\frac{\beta-\alpha}{2}}}\|_{\beta}^2=\frac{\Gamma(\beta+2)\Gamma(2\lambda-\alpha-2)}{[\Gamma(\lambda+\frac{\beta-\alpha}{2})]^2}\cdot\frac{1+{o}(1)}{(1-r^2)^{2\lambda-\alpha-2}},\, {\textup{as}}\,\, r\rightarrow 1^{-}.\nonumber 
 \end{equation}
It follows from Lemma \ref{lemma-key} that  
\begin{eqnarray}
\mathbb{M}_{g_1}^2(\alpha, \beta)&\geq&\mathbb{M}_{g_1(rz)}^2(\alpha, \beta)\geq \sup\limits_{r\in (0,1)}\frac{ \|2^{(\alpha-\beta)/2}{(1-rz)^{-\lambda-\frac{\beta-\alpha}{2}}}\|_{\beta}^2}{ \|{(1-rz)^{-\lambda}}\|_{\alpha}^2}\nonumber \\
&\geq& \frac{1}{2^{\beta-\alpha}}\frac{\Gamma(\beta+2)}{\Gamma(\alpha+2)}\frac{[\Gamma(\lambda)]^2}{[\Gamma(\lambda+\frac{\beta-\alpha}{2})]^2}. \nonumber\end{eqnarray} 
Then, (\ref{remark-1}) follows by letting $\lambda\to (\frac{\alpha}{2}+1)^{+}$. This proves (\ref{remark-0}).
\end{remark}

\begin{remark}Let $g\in \mathcal{A}(\Delta)$. The multiplication operators are related to the following Volterra-type integral operators $I_g$ and $J_g$, which are defined by
$$
I_{g}(\phi)(z) = \int_{0}^{z}\phi(\zeta)g'(\zeta)\,d\zeta,\,\, \phi\in \mathcal{A}(\Delta),$$
and$$
J_{g}(\phi)(z) = \int_{0}^{z}\phi'(\zeta)g(\zeta)\,d\zeta,\,\, \phi\in \mathcal{A}(\Delta).
$$
It is easy to see that 
$$
I_{g}(\phi)(z) + J_{g}(\phi)(z) = M_{g}(\phi)(z) - g(0)\phi(0).
$$
For $\alpha>-1$, we recall that the Dirichlet-type space ${\bf{D}}_{\alpha}^2={\bf{D}}_{\alpha}^2(\Delta)$ is defined by
$${\bf D}_{\alpha}^2(\Delta)=\{\phi \in \mathcal{A}(\Delta) : \|\phi\|_{{\bf D}_{\alpha}^2}^2:=|\phi(0)|^2+(\alpha+1)\iint_{\Delta}|\phi'(z)|^2(1-|z|^2)^{\alpha}\frac{dxdy}{\pi}<\infty\}.$$
Then we see that $\phi \in {\bf D}_{\alpha}^2$ if and only if $\phi' \in {\bf A}_{\alpha}^2$. 
From this, we can restate Theorem \ref{th-1} equivalently as
\begin{proposition}\label{pr-1}Let $\beta>\alpha>-1$. 
The operator $I_g$ is bounded from ${\bf A}_{\alpha}^2$ to ${\bf D}_{\beta}^2$ if and only if $g$ belongs to the Bloch-type space $\mathfrak{B}_{\gamma}$ with $\gamma=(\beta-\alpha)/2.$ 
Here, for $\gamma \in \mathbb{R}$, the Bloch-type space $\mathfrak{B}_{\gamma}=\mathfrak{B}_{\gamma}(\Delta)$ is defined by
$$
\mathfrak{B}_{\gamma}(\Delta)=\left\{ \phi \in \mathcal{A}(\Delta) : 
\|\phi\|_{\mathfrak{B}_{\gamma}}:=\sup_{z\in \Delta}|\phi'(z)|(1-|z|^2)^{\gamma}<\infty \right\}.
$$
\end{proposition}Note that $g\in \mathfrak{B}_{\gamma}$ if and only if $\phi'\in \mathcal{B}_{\gamma}$. When $I_g$ is bounded from ${\bf A}_{\alpha}^2$ to ${\bf D}_{\beta}^2$,  the norm of the operator $I_g$ about the symbol $g$, denoted by $\|g\|_{I}$, is defined by
$$
\|g\|_{I}:=\sup_{\phi \in {\bf A}^{2}_{\alpha}, \phi \neq 0} \frac{\|I_g(\phi)\|_{{\bf D}_{\beta}^2}}{\|\phi\|_{\alpha}}.
$$ 

Theorem \ref{th-1} also can equivalently be present as 
\begin{proposition}\label{pr-2}Let $\beta>\alpha>-1$. 
The operator $J_g$ is bounded from ${\bf D}_{\alpha}^2$ to ${\bf D}_{\beta}^2$ if and only if $g$ belongs to the growth space $\mathfrak{B}_{\gamma}$ with $\gamma=(\beta-\alpha)/2.$ 
\end{proposition}
When $J_g$ is bounded from ${\bf D}_{\alpha}^2$ to ${\bf D}_{\beta}^2$,  the norm of the operator $J_g$ about the symbol $g$, denoted by $\|g\|_{J}$, is defined by
$$
\|g\|_{J}:=\sup_{\phi \in {\bf D}^{2}_{\alpha}, \phi \neq 0}\frac{\|J_g(\phi)\|_{{\bf D}_{\beta}^2}}{\|\phi\|_{{\bf D}_{\alpha}^2}}.
$$ 

In view of Proposition \ref{pr-2} and based on (\ref{remark-0}), we can check that
\begin{proposition}
Let $\beta>\alpha>-1$. Let $g_2(z)=(1-z)^{-\frac{\beta-\alpha}{2}}, z\in \Delta$. Then we have 
$$\|g_2\|_{J}^2=\frac{\Gamma(\beta+2)}{\Gamma(\alpha+2)}\Big[\frac{\Gamma(1+\alpha/2)}{\Gamma(1+\beta/2)}\Big]^2.$$
\end{proposition}

For $\beta>\alpha>-1$, let 
\begin{equation}\label{g-l}g_3(z):=\sum_{n=0}^{\infty}\frac{\Gamma(n+\frac{\beta-\alpha}{2})}{(n+1)!\Gamma(\frac{\beta-\alpha}{2})}z^{n+1},\, z\in\Delta.\end{equation}
From (\ref{l-eq-1}), we know that $g$ is analytic in $\Delta$ and 
\begin{equation} 
[g_3(z)]'=(1-z)^{-\frac{\beta-\alpha}{2}},\, z\in \Delta,\nonumber \end{equation}so that $g$ belongs to $\mathfrak{B}_{(\beta-\alpha)/2}.$
In particular, when $\beta-\alpha\neq 2$, we have $$g_3(z)=\frac{2}{\beta-\alpha-2}(1-z)^{1-\frac{\beta-\alpha}{2}},\,z\in\Delta.$$

Then, similarly, from Proposition \ref{pr-1} and (\ref{remark-0}), we also have
\begin{proposition}
Let $\beta>\alpha>-1$. Let $g_3$ be as in (\ref{g-l}). Then we have 
$$\|g_3\|_{I}^2=\frac{\Gamma(\beta+2)}{\Gamma(\alpha+2)}\Big[\frac{\Gamma(1+\alpha/2)}{\Gamma(1+\beta/2)}\Big]^2.$$
\end{proposition}
\end{remark}

\begin{remark}
By checking carefully the proof of Theorem \ref{m-1}, we see that if a univalent rational function $f\in \mathcal{S}$ has a pole of order $2$, then $\mathbb{M}_f(\alpha)\geq \mathbb{M}_{\kappa}(\alpha)$ for each $\alpha>-1$. 
In \cite{Jin-c}, we have found some rational functions that induce the same multiplier norm as the Koebe function.
Next we will give more rational functions satisfying this characteristic. As some examples, we consider the functions $$f_b(z):=\frac{z+bz^2}{(1-z)^2},z\in \Delta,\,\, b\in [-1,0].$$ We next check that $f_b\in \mathcal{S}$ and $\mathbb{M}_f(\alpha)=\mathbb{M}_{\kappa}(\alpha)$ for each $\alpha>-1$.
Note that a simple computation gives $f_b(0)=f_b'(0)-1=0$. To prove $f_b\in \mathcal{S}$, we only need to show 
\begin{lemma}\label{ll-5}$f_b$ is univalent in $\Delta$ for each $b\in [-1,0]$.
\end{lemma}
\begin{proof}We prove it by contradiction. We assume that there exist $z_1,z_2\in\Delta$, $z_1\neq z_2$, such that $f_b(z_1)=f_b(z_2)$.
By a direct computation, we get that 
\begin{equation}\label{r-1}1+b(z_1+z_2)-(1+2b)z_1z_2=0.\end{equation}
If $b=-\frac{1}{2}$, then $1-\frac{1}{2}(z_1+z_2)=0$, this is impossible since $|z_1+z_2|<2.$ If $b=0$, then $1-z_1z_2=0$, this is impossible since $|z_1z_2|<1.$
When $b\in [-1,-\frac{1}{2})\cup (-\frac{1}{2},0)$, (\ref{r-1}) is equivalent to 
\begin{equation}(1+2b)[z_1z_2+\frac{b}{1+2b}(z_1+z_2)+1]=0.\nonumber\end{equation}
That is 
\begin{equation}\label{r-2}(z_1+\frac{b}{1+2b})(z_2+\frac{b}{1+2b})=\frac{(1-b)^2}{(1+2b)^2}.\end{equation}
If $b\in [-1,-\frac{1}{2})$ so that $\frac{b}{1+2b}>0$, then we have  
$$\Re{\Big[(z_1+\frac{b}{1+2b})(z_2+\frac{b}{1+2b})\Big]}<[1+\frac{b}{1+2b}]^2=\frac{(1+3b)^2}{(1+2b)^2},$$
for all $z_1,z_2 \in\Delta$. On the other hand, we have
$$(1-b)^2-(1+3b)^2=-8b(b+1)\geq 0,$$
so that (\ref{r-2}) is impossible. 

If $b\in (-\frac{1}{2},0)$ so that $\frac{b}{1+2b}<0$, then we have  
$$\Re{\Big[(z_1+\frac{b}{1+2b})(z_2+\frac{b}{1+2b})\Big]}<[-1+\frac{b}{1+2b}]^2=\frac{(1+b)^2}{(1+2b)^2},$$
for all $z_1,z_2 \in\Delta$. On the other hand, we have
$$(1-b)^2-(1+b)^2=-4b>0,$$
so that (\ref{r-2}) is also impossible. Combining all above arguments, we have proved Lemma \ref{ll-5}.  
 \end{proof}
We will check that 
\begin{lemma}\label{lemma-l}Let $b\in [-1, 0]$. For any $z\in \Delta$, we have
\begin{equation}\label{re-add-1}|S_{f_b}(z)|\leq |S_{\kappa}(z)|.\end{equation}
\end{lemma}
\begin{proof}
Note that
$$
S_{f_b}(z)=-\frac{6(1+b)^2}{\bigl[1+(1+2b)z\bigr]^2(1-z)^2}, \,\,\,
S_{\kappa}(z)=-\frac{6}{(1+z)^2(1-z)^2}.
$$ 
A direct computation yields that (\ref{re-add-1}) is equivalent to 
\begin{equation}\label{re-add-2}
|1+(1+2b)z|\geq (1+b)|1+z|, \,\text{for all}\,z\in \Delta. 
\end{equation}
Let $z=re^{i\theta}\in \Delta$. By a direct simple computation, we get that (\ref{re-add-2}) is equivalent to 
\begin{equation} 
r^2(3b^2+2b)-b^2-2b\geq 2b^2r\cos\theta,. \nonumber
\end{equation}for all $r\in (0,1), \, \theta\in (-\pi, \pi].$ Then, it is enough to prove that for each $b\in [-1,0]$,
\begin{equation}\label{h-1} 
r^2(3b^2+2b)-b^2-2b\geq 2b^2r,
\end{equation}for all $r\in (0,1).$ 

When $b=0$, (\ref{h-1}) is clearly true for all $r\in (0,1)$. When $b\in [-1,0)$, (\ref{h-1}) is equivalent to 
\begin{equation}\label{r-3} (r-1)[r(3b+2)+(b+2)]\leq 0, \end{equation}
for all $r\in (0,1)$. On the other hand, we note that, for $b\in [-1, -\frac{2}{3}]$, $$r(3b+2)+b+2\geq 4b+4 \geq 0,$$
for all $r\in (0,1)$ and for $b\in [-\frac{2}{3},0)$, $$r(3b+2)+b+2\geq b+2 \geq 0,$$
for all $r\in (0,1)$. These facts imply that for each $b\in [-1,0]$, (\ref{r-3}) is true for all $r\in (0,1)$ so that the lemma is proved. 
\end{proof}
Since Lemma \ref{lemma-l} implies that for any $b\in [-1,0]$, $\mathbb{M}_{f_b}(\alpha)\leq \mathbb{M}_{\kappa}(\alpha)$ for each $\alpha>-1$, combining above, we conclude that for any $b\in [-1,0]$, $\mathbb{M}_{f_b}(\alpha)=\mathbb{M}_{\kappa}(\alpha)$ for each $\alpha>-1$. 
\end{remark}
\begin{remark}
Although, Theorem \ref{m-1} tells us that there are many univalent functions $f$ satisfying that $\mathbb{M}_f(\alpha)\geq \mathbb{M}_{\kappa}(\alpha)$ for each $\alpha>-1$, also there are some functions satisfying that $\mathbb{M}_f(\alpha)=\mathbb{M}_{\kappa}(\alpha)$ for each $\alpha>-1$, we do not find any function $f$ for which $S_f$ induces a multiplier norm strictly bigger than that of the Koebe function. On the other hand, we have checked in \cite{Jin-1} that for each $\alpha>0$, $\mathbb{M}_f(\alpha)\leq \mathbb{M}_{\kappa}(\alpha)$ holds for a large class of univalent functions $f$ with a quasiconformal extension to the whole complex plane.  Hence we believe that the Koebe function 
induces the biggest multiplier norm for each $\alpha>0$. That is to say, we believe that the following conjecture is true, which implies the Brennan conjecture. \end{remark}
\begin{conjecture}\label{k}
For each $\alpha>0$, we have $\mathbb{M}_f(\alpha)\leq \mathbb{M}_{\kappa}(\alpha)$ for all $f\in \mathcal{S}$. 
\end{conjecture}
Let $\mathcal{S}_{*}$ be the class of all functions $f_{*}$ in $\mathcal{S}$ that have the form 
$$f_{*}(z)=\rho\circ \kappa \circ  \tau(z), z\in \Delta,$$
here $\rho$ is a M\"obius transformation, $\tau$ is a M\"obius transformation mapping the unit disk onto itself, and $\kappa$ is the Koebe function as before.  It is known that for each function $f_{*}\in \mathcal{S}_{*}$ there exists a curve in $\Delta$ such that $|S_{f_{*}}(z)|(1-|z|^2)^2=6$ for $z$ on that curve, while for any $f\in \mathcal{S}-\mathcal{S}_{*}$ we have $|S_f(z)|(1-|z|^2)^2<6$ for all $z\in \Delta$, see \cite[Page 60]{L}. On the other hand, we can check that for $\alpha>-1$, $\mathbb{M}_{f_{*}}(\alpha)=\mathbb{M}_{\kappa}(\alpha)$ for every $f_{*}\in \mathcal{S}_{*}$.
Indeed, for a function $f_* = \rho \circ \kappa \circ \tau$ belonging to $\mathcal{S}_{*}$, since the Schwarzian derivative is invariant under the Möbius transformation, we obtain $S_{f_{*}}=S_{\kappa\circ \tau}$.
To see $\mathbb{M}_{\kappa\circ\tau}(\alpha)=\mathbb{M}_{\kappa}(\alpha)$, we let $\tau(z)=e^{i\theta}(z-a)/{(1-\bar{a}z)}, a\in \Delta, z\in \Delta$, and define the operator $T$ on $\mathbf{A}_\alpha^2$ by
$(T\phi)(z) = \phi(\tau(z))[\tau'(z)]^{(\alpha+2)/2}$. It follows from a direct change of variables that $T$ is an isometric isomorphism of $\mathbf{A}_\alpha^2$ onto itself.
For any $\phi\in\mathbf{A}_\alpha^2$, set $\psi = T^{-1}\phi$, then $\|\phi\|_\alpha = \|\psi\|_\alpha$ and again a change of variables yields
$
\|S_{\kappa\circ\tau}\phi\|_{\alpha+4} = \|S_{\kappa} \psi\|_{\alpha+4}.
$
Hence $\mathbb{M}_{\kappa\circ\tau}(\alpha)=\mathbb{M}_{\kappa}(\alpha)$ and so that $\mathbb{M}_{f_*}(\alpha)=\mathbb{M}_{\kappa}(\alpha)$ for each $\alpha>-1$. These facts suggest that the above conjecture may be true. 

\begin{remark}Let $f_n\in \mathcal{S}, n\in\mathbb{N}$ be such that the sequence $f_n$ converges uniformly to $f\in \mathcal{S}$ on compact subsets of the unit disk. Then we see that $S_{f_n}$ also converges uniformly to $S_f$ on compact subsets of the unit disk.
By using the similar arguments as in the proof of Sub-Lemma \ref{sub-2}, we can prove that 
\begin{claim}\label{jia-p}
For $\alpha>-1$ and a fixed $\phi \in \mathbf{A}_{\alpha}^2$,  $$
\lim_{n \to \infty} \| S_{f_n}(z)\phi(z) - S_f(z)\phi(z) \|_{\alpha+4} = 0.$$
\end{claim}
Let $\alpha>-1$, we define the functional $\Pi_{\alpha}$ as
$$ \Pi_{\alpha}: f  \mapsto \mathbb{M}_f(\alpha),\,\, f\in \mathcal{S}.$$ 
Claim \ref{jia-p} implies that 
\begin{proposition}\label{jia-p-1}
For each $\alpha>-1$, the functional $\Pi_{\alpha}$ is continuous on the class $\mathcal{S}$ under the locally uniformly convergence topology. 
\end{proposition}
Since $\mathcal{S}$ is compact under the locally uniformly convergence topology, we obtain that   
\begin{proposition}\label{jia-p-2}
For each $\alpha>-1$, the functional $\Pi_{\alpha}$ attains at least one maximum in $\mathcal{S}$.
\end{proposition}
From Proposition 7.30 in \cite{Jin-c}, we know that for each $\alpha>-1$, the maximums of the functional $\Pi_{\alpha}$ can not be from the class $\mathcal{S}_{Q}$. 
Here, $\mathcal{S}_{Q}$ is the class of all univalent functions $f$ which belong to $\mathcal{S}$ and admit a quasiconformal extension to $\widehat{\mathbb{C}}$.  Proposition \ref{jia-p-1} tells us that, to prove Conjecture \ref{k}, it is enough to verify the conjecture on a dense subset of $\mathcal{S}$, such as the class of univalent polynomials in the unit disk.  
\end{remark}
\begin{remark}
Finally, we end the paper with the following
\begin{question}Under the notation and assumptions of Theorem \ref{m-2}, is it true that
$$\mathbb{M}(\alpha, \beta)=\mathbb{M}_{g_0}(\alpha,\beta),$$
for some $\alpha, \beta$? Here $g_0$ is defined as in Theorem \ref{m-2}.
\end{question}\end{remark}

\section{{\bf Acknowledgements}}
The author was supported by the National Natural Science Foundation of China (Grant No. 11501157).
\begin{spacing}{1.2}

\end{spacing}

\begin{thebibliography}{99}

\bibitem{AAR}Andrews G., Askey R.,  Roy R.,  {\em Special functions},  Cambridge University Press, Cambridge, 1999.

\bibitem{Be}Bertilsson D., {\em On Brennan's conjecture in conformal mapping}, Dissertation, Royal Institute of Technology, 1999.

\bibitem{Du} Duren P., {\em Univalent functions}, Springer-Verlag, New York, 1983.

\bibitem{DGM}Duren P., Gallardo-Gutiérrez E., Montes-Rodríguez A., {\em A Paley-Wiener theorem for Bergman spaces with application to invariant subspaces}, Bull. Lond. Math. Soc., 39 (2007), no. 3, pp. 459-466.

\bibitem{GM}Garnett J., Marshall D., {\em Harmonic Measure}, New Math. Monogr., vol. 2, Cambridge Univ. Press, Cambridge, 2005.

   
\bibitem{HSo}Hedenmalm H., Sola A., {\em Spectral notions for conformal maps: a survey}, Comput. Methods Funct. Theory,  8 (2008), no. 1-2, pp. 447-474.

\bibitem{HS-1}Hedenmalm H., Shimorin S., {\em Weighted Bergman spaces and the integral means spectrum of conformal mappings},  Duke Math. J. 127 (2005), no. 2, pp. 341-393.

\bibitem{HP}Hille E., Phillips R., {\em Functional analysis and semi-groups}, Amer. Math. Soc. Colloq. Publ., Vol. 31,  American Mathematical Society, Providence, RI, 1957. xii+808 pp.

\bibitem{Jin-1}Jin J., {\em On a multiplier operator induced by the Schwarzian derivative of univalent functions}, Bull. Lond. Math. Soc., 56 (2024), no. 7, pp. 2296-2314.

\bibitem{Jin-c}Jin J., {\em Complex exponential integral means spectra of univalent functions and the Brennan conjecture}, arXiv: 2512.09330.

\bibitem{L}Lehto O., {\em Univalent functions and Teichm\"uller spaces}, Springer-Verlag, 1987.

\bibitem{SS}Shimorin S., {\em A multiplier estimate of the Schwarzian derivative of univalent functions}, Int. Math. Res. Not., No. 30, 2003.

\bibitem{zhao}Zhao R., {\em  Pointwise multipliers from weighted Bergman spaces and Hardy spaces to weighted Bergman spaces}, Ann. Acad. Sci. Fenn. Math., 29 (2004), no. 1, pp. 139-150.
\end{thebibliography}
\end{document}